\documentclass[reqno,a4paper]{amsart}
\usepackage[english,activeacute]{babel}
\usepackage{amssymb,amsmath,amsthm,amsfonts,mathrsfs,hyperref,mathtools,stmaryrd, dsfont}
\DeclareFontFamily{U}{mathc}{}
\DeclareFontShape{U}{mathc}{m}{it}%
{<->s*[1.03] mathc10}{}

\usepackage{authblk}

\DeclareMathAlphabet{\mathscr}{U}{mathc}{m}{it}
\usepackage[font=footnotesize]{caption}
\usepackage{tikz}
\usetikzlibrary{intersections,shapes,trees,arrows,spy,positioning,decorations,arrows.meta,decorations.pathreplacing,patterns,calc}
\usepackage{tikz-cd}
\usepackage{tkz-euclide}
\usepackage{pgfplots}
\usepackage{xcolor}
\usepackage[numbers,square,sort]{natbib}
\usepackage{tabularx}
\usepackage{enumitem}  
\usepackage[T1]{fontenc}
\usepackage[foot]{amsaddr}
\usepackage{url}  

\theoremstyle{plain}
\newtheorem{theorem}{Theorem}[section]

\newtheorem{lemma}[theorem]{Lemma}
\newtheorem{proposition}[theorem]{Proposition}
\newtheorem{corollary}[theorem]{Corollary}
\theoremstyle{definition}

\newtheorem{definition}[theorem]{Definition}

\theoremstyle{remark}
\newtheorem{remark}[theorem]{Remark}

\numberwithin{equation}{section}

\DeclareMathOperator{\poi}{Poi}
\DeclareMathOperator{\expo}{Exp}

\newcommand{\defeq}{\mathrel{\mathop:}=}

\usepackage{tikzscale}

\definecolor{darkgreen}{rgb}{.125,.5,.25}
\definecolor{darklavender}{rgb}{0.5, 0, 0.5}
\definecolor{navyblue}{rgb}{0.0,0.14,0.4}

\newcommand{\dg}[1]{\textcolor{darkgreen}{#1}}

\date{\today}%

\begin{document}
	
	\title{Critical multitype branching: a functional limit theorem and ancestral processes}
	
	\author{Ellen Baake$^{(1)}$}
	\address[1]{Faculty of Technology, Bielefeld University, Box 100131, 33501 Bielefeld, Germany}
	\email{ebaake@techfak.uni-bielefeld.de}
	
	\author{Fernando Cordero$^{(2)}$}
	\address[2]{Department of Natural Sciences and Sustainable Ressources, Institute of Mathematics, BOKU University, Greger-Mendel-Strasse 33/II, 1180 Vienna, Austria}
	\email{fernando.cordero@boku.ac.at}
	
	\author{Sophia-Marie Mellis$^{(1)}$}
	\email{smellis@techfak.uni-bielefeld.de}
	
	\author{Vitali Wachtel$^{(3)}$}
	\address[3]{Faculty of Mathematics, Bielefeld University, Box 100131, 33501 Bielefeld, Germany}
	\email{wachtel@math.uni-bielefeld.de}

	\begin{abstract}
		We  consider the long-term behaviour of critical multitype branching processes conditioned on non-extinction, both with respect to the forward and the ancestral processes. Forward in time, we  prove a functional limit theorem in the space of trajectories of the linearly-scaled $h$-transformed process; the change of measure allows us to work on the same probability space for all times. Backward in time, we trace the lines of descent of individuals sampled from a (stationary) population and analyse the ancestral type distribution, that is, the type distribution of the ancestors in the distant past, as well as the type process along the ancestral line.
	\end{abstract}
	
	\maketitle

    
	
	
	\section{Introduction}
	
	Multitype branching processes are widely used as prototype models for a variety of phenomena, such as the spread of infectious diseases in  populations with heterogeneous susceptibilities, 
	cell-kinetic studies with multiple cell types, or the development of complex networks in social systems, see, for example, \cite{Keating2022, Yakovlev2009, Sarabjeet2014,  Spencer_ONeill_2011, Penisson_2014}.
	Classical results concern the forward time evolution in the long term, such as extinction probabilities, asymptotic growth rates, and stationary type distributions conditional on nonextinction \cite{mullikin1963limiting,joffe1967multitype,athreya2004branching}. In the supercritical case, this has been complemented by the backward point of view, which accounts for the evolutionary history of  individuals  \cite{jagers1989,jagers1992stabilities,jagers1996asymptotic,georgii2003supercritical}.  Here, one traces back the  line of descent of individuals sampled from a (stationary) population and analyses the \emph{ancestral type distribution}, that is, the type distribution of the ancestors in the distant past, as well as the type process along the ancestral line.  The resulting connections between the present and the past have led to important insight in population biology and the theory of biological evolution, see \cite{SughiyamaKobayashi2017,Smerlak2021,CheekJohnston2023} for some recent examples. \\
	
	In the supercritical case, the long-term behaviour both forward and backward in time is solely determined by the Perron--Frobenius eigenvalue and eigenvectors of the first-moment matrix of the process, see \cite[Thm.~V.6.1]{athreya2004branching} and \cite{georgii2003supercritical}. Characteristically, one has almost-sure convergence to the stationary type distributions (conditional on nonextinction). This concentration result is related to the dichotomy of either extinction or exponential growth, which gives access to the strong law of large numbers immediately. \\

	The purpose of this paper is to extend this analysis to the critical case. 
	Typically, critical branching processes conditioned on non-extinction (either by `simple' conditioning on survival up to time $n$ or through Doob's $h$-transform) exhibit linear growth over time. 
	It is well known that the linearly-scaled process converges weakly to a \emph{non-degenerate} limiting distribution, see  \cite{joffe1967multitype,mullikin1963limiting} or \cite[Thm.~V.5.1]{athreya2004branching}. This is once more related to the left Perron--Frobenius eigenvector of the first moment matrix, but there is no concentration as in the supercritical case.  \\
	
	In this contribution, we will characterise the  limiting forward-time \emph{process} by means of a Bessel process, and investigate the ancestral type distribution and the type process along the ancestral line. Surprisingly, it will turn out that, despite the very different behaviour forward in time, the long-term behaviour of the ancestral process is similar to that in the supercritical case; in particular, it still follows the law of large numbers, but this time the weak law as opposed to the strong one. 
	\\   
	
	The paper is organised as follows. After setting up the family tree construction and the change of measure that will be our workhorse (Section~\ref{prelim}), we prove the announced functional limit theorem in the space of trajectories of the $h$-transformed process (Section~\ref{seclimittheorem}). The change of measure allows us to work on the same probability space for all times. On the way, we obtain explicit expressions for the entrance law and the limiting transition probabilities, both under the original and the transformed measures; this will already demonstrate that  the transformed measure is linked to size-biasing of certain random variables. Next, we establish 
	laws of large numbers for genealogical quantities. Specifically, as in \cite{georgii2003supercritical}, we consider individuals alive at some time $n$ and examine the types of their ancestors at an earlier time $n - m$. As both $n$ and $m$ tend to infinity, we show that  the population average of the ancestral types converges to a distribution that is related to the left and right Perron--Frobenius eigenvectors of the first-moment matrix in the same way as in the supercritical case (Section~\ref{secpopulationaverage}). A similar result holds for the corresponding time average along the ancestral line (Section~\ref{sectimeaverage}). The latter result relies on the size-biased family tree of \cite{kurtz1997conceptual}, which we construct in Section~\ref{secsizebiasing}, where the connection with the transformed measure will also become clear. Finally, in Section~\ref{theolargedevtypicalancestral}, we establish the law of the type process along a typical ancestral line.

	\section{Preliminaries}\label{prelim}
	\subsection{The process and the family tree}\label{secfamilytree}
	We adopt most of the notation in \cite{georgii2003supercritical}: Let $[d]= \{1, \dots, d\}$ be a finite set of types. An individual of type $i \in [d]$ lives for one unit of time, which we refer to as \emph{generation}, and then splits into a random offspring $\tilde{\boldsymbol{N}}_{i} = (\tilde{N}_{i,j})_{j \in [d]}$  (we will get rid of the tilde soon) with distribution $\boldsymbol{p}_{i}$ on $\mathbb{N}_{0}^{d}$ and finite means $m_{i,j}\coloneq \mathbb{E}[\tilde{N}_{i,j}]>0$ for all $i,j \in [d]$. Here, $\tilde{N}_{i,j}$ is the number of $j$-children, and $\mathbb{N}_{0}= \{ 0, 1, 2, \dots\}$. \\
	
	We assume that the mean offspring matrix $\mathbf{M}= (m_{i,j})_{i,j \in [d]}$ is irreducible. Perron-Frobenius (PF) theory therefore tells us that $\mathbf{M}$ has a principal eigenvalue $\lambda$ associated with positive left and right eigenvectors  $\boldsymbol{v}$ and $\boldsymbol{u}$, which will be normalised such that 
	\begin{equation}\label{our_normalisation}
	   \langle \boldsymbol{1}, \boldsymbol{v} \rangle = 1 = \langle \boldsymbol{u}, \boldsymbol{v} \rangle,
	\end{equation}
	where $\boldsymbol{1} \defeq ({1, \ldots, 1})$. We will throughout consider the critical case $\lambda =1$.
	
	\begin{remark}\label{remnormJS}
	Many classical references, such as \cite{joffe1967multitype}, \cite{mullikin1963limiting}, and \cite[Chap.~5]{athreya2004branching} use left and right Perron--Frobenius eigenvectors $ \tilde{\boldsymbol{v}}$ and $\tilde{\boldsymbol{u}}$ normalised so that $\langle \boldsymbol{1}, \tilde{\boldsymbol{u}} \rangle = 1 = \langle \tilde{\boldsymbol{u}}, \tilde{\boldsymbol{v}} \rangle$, in line with a standard formulation of Perron--Frobenius theory (see, for example, \cite[App.~2,~Thm.~2.3]{karlintaylor1975}). That is, they are related to ours via
		\[
		   \boldsymbol{v} = \frac{\tilde{\boldsymbol{v}}}{\langle \boldsymbol{1}, \tilde{\boldsymbol{v}} \rangle}  \quad \text{and} \quad \tilde{\boldsymbol{u}} = \frac{\boldsymbol{u}}{\langle \boldsymbol{1}, \boldsymbol{u} \rangle},
		\]
		and, due to \eqref{our_normalisation}, $\langle \tilde{\boldsymbol{u}}, \tilde{\boldsymbol{v}} \rangle =1$ implies
		\begin{equation}  
		    \langle \boldsymbol{1}, \tilde{\boldsymbol{v}} \rangle = \langle \boldsymbol{1}, \boldsymbol{u} \rangle. \label{eqnromJS}
		\end{equation}

		We use our normalisation since it makes both $\boldsymbol{v}$ and $\boldsymbol{\alpha}\defeq (u_{i}v_{i})_{i \in [d]}$, the stationary type distributions forward and backward in time, respectively, into properly-normalised probability distributions.
	\end{remark} 
	
	According to \cite{harris1963theory}, the associated random family tree can be constructed as follows. Let $\mathbb{X} = \bigcup_{n \in \mathbb{N}_{0}} \mathbb{X}_{n}$, where $\mathbb{X}_{n}$ describes the potential individuals of the $n$'th generation. That is, $\mathbb{X}_{0}=[d]$ and $i_{0} \in \mathbb{X}_{0}$ specifies the type of the root, i.e. the founding ancestor. Next, $\mathbb{X}_{1}= [d] \times \mathbb{N}$, and the element $x=(i_{1}, \ell_{1})$ is the $\ell_{1}$'th $i_{1}$-child of the root. Finally, for $n>1$, $\mathbb{X}_{n}= ([d] \times \mathbb{N})^{n}$, and $x=(i_{1}, \ell_{1}; \ldots; i_{n}, \ell_{n})$, $x$ is the $\ell_{n}$'th $i_{n}$-child of its parent $\tilde{x} = (i_{1}, \ell_{1}; \ldots; i_{n-1}, \ell_{n-1})$. We write $\sigma(x)$ for the type of $x \in \mathbb{X}_{n}$, that is, if $x=(i_{1}, \ell_{1}; \ldots; i_{n}, \ell_{n})$ then $\sigma(x)=i_{n}$. \\
	
	With each $x \in \mathbb{X}$ we associate its random offspring $\boldsymbol{N}_{x} = (N_{x,j})_{j \in [d]}\in \mathbb{N}_{0}^{d}$ with distribution $\boldsymbol{p}_{\sigma(x)}$, where we assume the family $\{\boldsymbol{N}_{x} : x \in \mathbb{X}   \}$ to be independent. The random variables $\boldsymbol{N}_{x}$ indicate which of the virtual individuals $x \in \mathbb{X}$ are actually realised, namely those in the random set $X= \bigcup_{n \in \mathbb{N}_{0}}X_{n}$ defined recursively by \[ X_{0}= \{i_{0}\}, \quad X_{n}= \{ x=(\tilde{x}; i_{n}, \ell_{n}) \in \mathbb{X} : \tilde{x} \in X_{n-1}, \ell_{n}\leq N_{\tilde{x}, i_{n}}   \}, \quad \text{for }n\geq 1,         \] where $i_{0}$ is the prescribed type of the root. \\
	
	The family tree is completely determined by the process $(X_{n})_{n \in \mathbb{N}_{0}}$. For $0<m<n$ and $x \in X_{n}$ we write $x(m)$ for the unique ancestor of $x$ living in generation $m$, that is, the unique $x(m) \in X_{m}$ such that $x=x(m)y$ for some $y \in ([d \times \mathbb{N})^{n-m}$. The concatenation $xy$ of two elements $x,y \in \mathbb{X}$ is defined in the natural way. On the other hand, for $x \in X_{m}$ we let \begin{equation} X(x, n) = \{ y \in \mathbb{X}_{n-m}: xy \in X_{n}  \}  \label{defsetdescendants}
	\end{equation} denote the set of descendants of $x$ living in generation $n$ (see Figure~\ref{figbranchingprocess}). The descendant trees $(X(x, n))_{n \geq m}$ with $x \in X_{m}$ are conditionally independent given $(X_{k})_{k \in [m]}$, with distribution $\mathbb{P}_{\sigma(x)}$, where $\mathbb{P}_{i}$ denotes the distribution of $(X_{n})_{n \in \mathbb{N}_{0}}$ on $\Omega$ when the type of the root is $i_{0}=i$. \\
	
	We will be interested in  \[ \boldsymbol{Z}(n) \coloneq (Z_{i}(n))_{i \in [d]}, \quad \text{where} \quad Z_{i}(n)= |\{ x \in X_{n} : \sigma(x)=i  \}|  \] denotes the number of type-$i$ individuals alive in generation $n$. In particular, the process $(\boldsymbol{Z}(n))_{n \in \mathbb{N}_{0}}$ is a multitype Galton-Watson branching process, which we may also describe recursively via \begin{equation} Z_{i}(n+1) = \sum_{\ell=1}^{d}  \sum_{j=1}^{Z_{\ell}(n)}  \xi_{\ell,i}^{(j)}(n), \quad i\in [d] \label{eqdefgaltonwatson}
	\end{equation} where $ \xi_{\ell,i}^{(j)}(n)$ denotes the number of type-$i$ offspring of the $j$-th type $\ell$ individual in generation $n$ and the $( \xi_{\ell,i}^{(j)}(n))_{j, n \in \mathbb{N}_{0}}$ are i.i.d. From now on, we will often consider collections of family trees and write $\mathbb{P}_{\boldsymbol{z}}$ for the distribution of $(\boldsymbol{Z}(n))_{n \in \mathbb{N}_{0}}$, started from $\boldsymbol{Z}(0)=\boldsymbol{z}$.
	
	\begin{figure}[h]
		\includegraphics[width=0.8\textwidth]{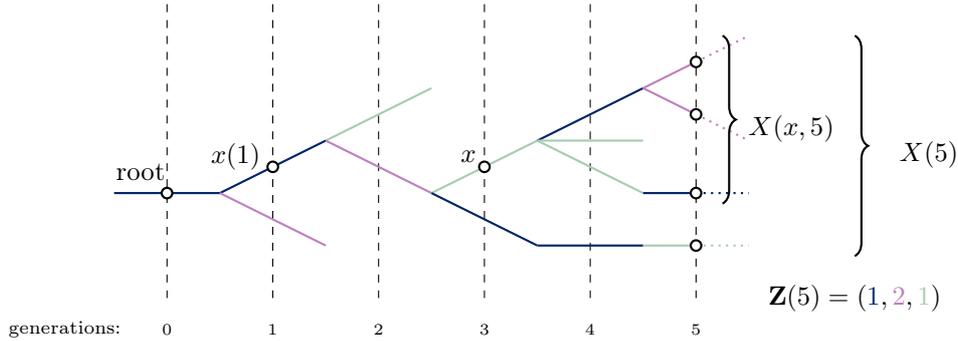}
		\caption{A family tree for $d=3$. Types are indicated by different colours, indexed in the order (\textcolor{navyblue}{blue}, \textcolor{darklavender!70}{magenta}, \textcolor{darkgreen!60}{green}), counted from top to bottom. We have distinguished the 3rd-generation individual $x=(\textcolor{navyblue}{\text{blue}}, 1; \textcolor{darklavender!70}{\text{ magenta}}, 1; \textcolor{darkgreen!60}{\text{ green}}, 1)$ (that is, the circled edge intersecting the vertical dashed line at $3$). Its unique ancestor in generation 1 is $x(1)=(\textcolor{navyblue}{\text{blue}}, 1)$; $X(x,5)$ (that is, the set of all edges that emanate from $x$ and hit the vertical dashed line at $5$) is its set of descendants in generation 5, and $\boldsymbol{Z}(5)$ counts the type frequencies in the population $X(5)$ (that is, all edges that intersect the vertical dashed line at $5$).}
		\label{figbranchingprocess}
	\end{figure}
	In particular, we may now also write the mean offspring matrix $\mathbf{M}$ in terms of $(\boldsymbol{Z}(n))_{n \in \mathbb{N}_{0}}$: \[ m_{i,j} = \mathbb{E}_{\boldsymbol{e}_{i}}[Z_{j}(1) ], \quad i,j \in [d].   \]
	
	Since we consider only the critical case, we have $\mathbb{P}(\Omega_{\mathrm{surv}}(n)) \underset{n \to \infty}{\longrightarrow} 0$, where we write \[  \Omega_{\mathrm{surv}}(n) \coloneq \{ X_{n} \neq \varnothing \}   \] for the event that the population survives until time $n$. Conditioned on this event, it is well known that the population grows linearly with $n$ (at least asymptotically). More specifically, according to \cite[Thm.~6]{joffe1967multitype} and \cite[Thm.~8]{mullikin1963limiting}, we have, under $\mathbb{P}_{\boldsymbol{z}}$,  \begin{equation} \frac{\boldsymbol{Z}(n)}{n} \ \Big{\vert} \ \Omega_{\mathrm{surv}}(n) \Rightarrow W \boldsymbol{v} \quad \text{as } \, n \to \infty, 
	\label{resjoffespitzer}
	\end{equation} where  $\Rightarrow$ denotes weak convergence and $W \sim \expo(1/Q[\boldsymbol{u}])$ is an exponential random variable with parameter $1/Q[\boldsymbol{u}]$. Here, $Q[\cdot]$ denotes the quadratic form defined by $Q[\boldsymbol{s}] \coloneq \langle \boldsymbol{v}, \boldsymbol{q}[\boldsymbol{s}] \rangle$, where \[ q_{i}[\boldsymbol{s}] =	\frac{1}{2} \sum_{k=1}^{d} \sum_{j=1}^{d} s_{j} s_{k} \mathbb{E}_{\boldsymbol{e}_{i}} [Z_{k}(1)Z_{j}(1) - \delta_{j,k} Z_{k}(1)] < \infty    \] for $\boldsymbol{s}= (s_{i})_{i \in [d]}$ and $s_{i} \in [0,1]$. 
	
	\begin{remark}\label{remsqrtd}
		Actually, in \cite[Thm.~6]{joffe1967multitype} and \cite[Thm.~8]{mullikin1963limiting} (but not in \cite[Thm.~V.5.1]{athreya2004branching}), the parameter of the exponential distribution includes an additional factor $\sqrt{d}$. However, examining the proofs makes clear that this is in fact a misprint.
	\end{remark}
	
	\begin{remark}\label{remQ}
		Note that our definition of the quadratic form $Q[\cdot]$ differs from that in \cite{joffe1967multitype,mullikin1963limiting} and \cite[Chap.~V.2]{athreya2004branching}, although it looks very similar: due to their normalisation (cf.\ Remark~\ref{remnormJS}), their choice of $Q[\cdot]$, here denoted by $\widetilde{Q}[\cdot]$, translates to our setting as 
		\begin{equation}\label{Qtilde} \widetilde{Q}[\boldsymbol{s}]  \coloneq \langle  \tilde{\boldsymbol{v}}, \boldsymbol{q}[\boldsymbol{s}] \rangle = \langle \boldsymbol{1}, \tilde{\boldsymbol{v}} \rangle  Q[\boldsymbol{s}].    
		\end{equation}
		The other way around, they write \eqref{resjoffespitzer} as, under $\mathbb{P}_{\boldsymbol{z}}$, \[ \frac{\boldsymbol{Z}(n)}{n} \ \Big{\vert} \ \Omega_{\mathrm{surv}}(n) \Rightarrow \widetilde{W} \tilde{\boldsymbol{v}} \quad \text{as } \, n \to \infty,   
		\] with $\widetilde{W} \sim \expo\big(\sqrt{d}/\widetilde{Q}[\tilde{\boldsymbol{u}}]\big)$, so the parameter of the exponential distribution, up to $\sqrt{d}$ (see Remark~\ref{remsqrtd}), translates from their notation to ours as \[ \frac{1}{\langle \boldsymbol{1}, \tilde{\boldsymbol{v}} \rangle \widetilde{Q}[\tilde{\boldsymbol{u}}]} = \frac{1}{\langle \boldsymbol{1}, \tilde{\boldsymbol{v}} \rangle^{2} Q[\frac{\boldsymbol{u}}{\langle \boldsymbol{1}, \boldsymbol{u} \rangle}]} = \frac{1}{Q[\boldsymbol{u}]},  \] 
		where the first step is due to \eqref{Qtilde}. Therefore, in the light of Remarks~\ref{remnormJS} and \ref{remsqrtd}, $\widetilde{W} \tilde{\boldsymbol{v}} \overset{d}{=} W \boldsymbol{v}$, so that \eqref{resjoffespitzer} corresponds to the results of \cite{joffe1967multitype} and \cite{mullikin1963limiting}. 
	    In particular, according to \cite[(3.14)]{joffe1967multitype}: \begin{equation}\label{eqlimitOmegasurvn}
				n\mathbb{P}_{\boldsymbol{z}} (\Omega_{\mathrm{surv}}(n)) \underset{n \to \infty}{\longrightarrow} \frac{\langle \tilde{\boldsymbol{u}}, \boldsymbol{z} \rangle}{\widetilde{Q}[\tilde{\boldsymbol{u}}]} = \frac{\langle \boldsymbol{u}, \boldsymbol{z} \rangle}{\langle \boldsymbol{1}, \tilde{\boldsymbol{v}} \rangle Q\big[\frac{\boldsymbol{u}}{\langle \boldsymbol{1}, \boldsymbol{u} \rangle }\big] \langle \boldsymbol{1}, \boldsymbol{u} \rangle} = \frac{\langle \boldsymbol{u}, \boldsymbol{z}\rangle}{Q[\boldsymbol{u}]}. 
		\end{equation} 
	\end{remark}
	
	
	\subsection{The measure  $\widehat{\mathbb{P}}$}\label{secP}
	We now introduce the change of measure that will be a key ingredient to our results.
	
	\begin{definition}\label{defhatP}
		We define the probability measure $\widehat{\mathbb{P}}_{\boldsymbol{Z}(0)}$ as Doob's $h$-transform of $\mathbb{P}_{\boldsymbol{Z(0)}}$ with the harmonic function $h(\boldsymbol{x})\coloneq \langle \boldsymbol{u}, \boldsymbol{x} \rangle$: \[\widehat{\mathbb{P}}_{\boldsymbol{Z}(0)}(A(n)) \coloneq  \mathbb{E}_{\boldsymbol{Z}(0)} \bigg[ \frac{ \langle \boldsymbol{u}, \boldsymbol{Z}(n) \rangle}{\langle \boldsymbol{u}, \boldsymbol{Z}(0) \rangle   } \mathds{1}_{A(n)} \bigg], \quad A(n) \in \mathcal{F}_{n}:= \sigma\{ \boldsymbol{Z}(k) : k \leq n \}.     \]
	\end{definition} To see that $h$ is actually harmonic, we note that \begin{align*}
		\mathbb{E}[ h(\boldsymbol{Z}(n+1)) \ \vert \ \mathcal{F}_{n}   ] = \mathbb{E}[ \langle \boldsymbol{u}, \boldsymbol{Z}(n+1) \rangle \ \vert \ \mathcal{F}_{n}  ] &= \langle \boldsymbol{u}, \mathbb{E}[ \boldsymbol{Z}(n+1) \ \vert \ \mathcal{F}_{n} ] \rangle \\
		&= \langle \boldsymbol{u}, \boldsymbol{Z}(n)  \mathbf{M}\rangle = \langle \mathbf{M} \boldsymbol{u} , \boldsymbol{Z}(n) \rangle = \langle \boldsymbol{u}, \boldsymbol{Z}(n) \rangle = h(\boldsymbol{Z}(n)).
	\end{align*} 
	We will throughout use both the notations $h(\cdot)$ and $\langle \boldsymbol{u}, \cdot \rangle$ depending on whether we want to emphasise the harmonicity or the linearity of the function. \\
	
	One advantage of working with $\widehat{\mathbb{P}}$ is that it allows us to get rid of the conditioning on $\Omega_{\mathrm{surv}}(n)$ since \[  \mathbb{E} \bigg[ \frac{ \langle \boldsymbol{u}, \boldsymbol{Z}(n) \rangle}{\langle \boldsymbol{u}, \boldsymbol{Z}(0) \rangle   } \mathds{1}_{A(n)} \mathds{1}_{\Omega_{\mathrm{surv}}(n)}   \bigg] = \mathbb{E} \bigg[ \frac{ \langle \boldsymbol{u}, \boldsymbol{Z}(n) \rangle}{\langle \boldsymbol{u}, \boldsymbol{Z}(0) \rangle   } \mathds{1}_{A(n)}  \bigg].  \] We will use this property  to establish a functional limit theorem (Theorem~\ref{theobessel}) and, in Theorem~\ref{theopopaverageancestraltypes}, an analogue to \cite[Thm.~3.1]{georgii2003supercritical}  concerning the population average of the ancestral types. Some of these results are established under the measure $\widehat{\mathbb{P}}$, but it is straightforward to translate them back to analogous results for $\mathbb{P}(\cdot \ \vert \ \Omega_{\mathrm{surv}}(n))$, see Proposition~\ref{corobessel} and Theorem~\ref{coropopaverage}. \\

	On the other hand, some of our results in terms of $\widehat{\mathbb{P}}$ (in particular Propositions~\ref{propentrancelaw} and \ref{proptransitionprob}) will turn out to have natural probabilistic interpretations related to size biasing (see, in particular, Propositions~\ref{propentrancelaw} and \ref{proptransitionprob}). This is because $\widehat{\mathbb{P}}$ is closely connected to size-biased multitype Galton--Watson trees in the sense of \cite{kurtz1997conceptual}; they will be  introduced in Section \ref{secsizebiasing}.

		
		\section{A functional limit theorem}\label{seclimittheorem}
		In this section, we prove the announced fluctuation limit. Note that this result differs significantly from the supercritical case, where the Kesten--Stigum theorem \cite{KestenStigum1966} tells us that we have exponential growth to a concentration limit. In contrast, as with \eqref{resjoffespitzer},  growth is linear in the critical case, and one obtains a limit that is truly random: 
		\begin{theorem}\label{theobessel}
			Let $(B(t))_{t \geq 0}$ be a $4$-dimensional squared Bessel process  started at $B(0)=0$. Assume that the third moments of the offspring distribution are finite and set, for $T>0$,
			\begin{align*}  & \boldsymbol{Y}^{(n)}\coloneq (\boldsymbol{Y}^{(n)}(t))_{t \in [0,T]} \quad\text{with}\quad \boldsymbol{Y}^{(n)}(t) = \frac{\boldsymbol{Z}(\lfloor nt \rfloor)}{n Q[\boldsymbol{u}]} \\
				\text{and} \quad  & \boldsymbol{Y}\coloneq (\boldsymbol{Y}(t))_{t \in [0,T]} \quad\text{with}\quad \boldsymbol{Y}(t) = B(t) \boldsymbol{v}.
			\end{align*} Then, under $\widehat{\mathbb{P}}_{\boldsymbol{z}}$ and for any $\boldsymbol{z} \in \mathbb{N}_{0}^{d}\setminus \{ \boldsymbol{0}\}$, $\boldsymbol{Y}^{(n)}$ converges in distribution, as $n \to \infty$, to $ \boldsymbol{Y}$ on $\mathcal{D}([0,T], \mathbb{R}^{d}).$
		\end{theorem}

\begin{remark}\label{DankaPap}
In Section~\ref{secsizebiasing}, we describe a standard construction to represent a multitype branching process under the measure $\widehat{\mathbb{P}}$ as a size-biased tree, that is, a branching process with immigration combined with a so-called trunk or spine. In this construction, the law along the trunk differs from that of the remaining population, and the distribution of immigrants  in turn depends on the current type of the  trunk. The construction thus differs from a standard multitype branching process with immigration, which has no trunk, and where the immigrants are i.i.d. vectors. The paper by Danka and Pap \cite{Danka_Pap_16}, where an analogue of our Theorem~\ref{theobessel} was proved for the standard branching process with immigration, is therefore not directly applicable. It may be possible to adopt their arguments to our situation, but we prefer to give a direct proof for the $h$-transformed process, which is much simpler than the proof of Danka and Pap and does not require more assumptions than their approach. We also note in passing that a critical \emph{single-type}  $h$-transformed branching process is always
equivalent to a process with immigration.
\end{remark}

		The proof will be split into two parts: fdd convergence and characterisation of the limit (Section~\ref{secfdd}) and tightness (Section~\ref{sectightness}). But first we prove the following consequence of Theorem~\ref{theobessel}.
		\begin{proposition}\label{corobessel} For every bounded and continuous function $f: \mathcal{D}([0,1], \mathbb{R}^{d}) \rightarrow \mathbb{R}$, one has
			\[  \mathbb{E}_{\boldsymbol{z}} \big[ f (\boldsymbol{Y}^{(n)})  \ \big{\vert} \  \Omega_{\mathrm{surv}}(n)  \big] \underset{n \to \infty}{\longrightarrow} \mathbb{E}_{0} \bigg[  \frac{f(\boldsymbol{Y})}{h(\boldsymbol{Y}(1))}   \bigg],    \]where $\boldsymbol{Y}^{(n)}$ and $\boldsymbol{Y}$ are defined as in Theorem~\ref{theobessel} and the expectation on the right-hand side is taken with respect to the law of the Bessel process defined in Theorem~\ref{theobessel}.
		\end{proposition}
		\begin{proof}
			Without loss of generality, we may assume $f: \mathcal{D}([0,1], \mathbb{R}^{d}) \rightarrow [0,1]$ instead of $f: \mathcal{D}([0,1], \mathbb{R}^{d}) \rightarrow \mathbb{R}$. Then, by Definition~\ref{defhatP}, we find that
			\begin{align}
				\mathbb{E}_{\boldsymbol{z}}[ f(\boldsymbol{Y}^{(n)}) \ \vert \ \Omega_{\mathrm{surv}}(n) ] &= \frac{1}{\mathbb{P}_{\boldsymbol{z}}(\Omega_{\mathrm{surv}}(n))} \ \mathbb{E}\big[ f(\boldsymbol{Y}^{(n)}) \mathds{1}_{\{ \Omega_{\mathrm{surv}}(n) \}}     \big] \nonumber  \\
				&= \frac{h(\boldsymbol{z})}{\mathbb{P}_{\boldsymbol{z}}(\Omega_{\mathrm{surv}}(n))} \ \frac{1}{h(\boldsymbol{z})} \ \mathbb{E}\bigg[ \frac{h(\boldsymbol{Z}(n))}{h(\boldsymbol{Z}(n))} \ f(\boldsymbol{Y}^{(n)}) \mathds{1}_{\{ \Omega_{\mathrm{surv}}(n) \}}  \bigg] \nonumber \\ 
				&=  \frac{h(\boldsymbol{z})}{\mathbb{P}_{\boldsymbol{z}}(\Omega_{\mathrm{surv}}(n))} \ \widehat{\mathbb{E}}_{\boldsymbol{z}} \bigg[ \frac{f(\boldsymbol{Y}^{(n)})}{h(\boldsymbol{Z}(n))}     \bigg]. \label{VWUmleitung1}
			\end{align} Since $h(\boldsymbol{Z}(n))= n h(\boldsymbol{Y}^{(n)}(1))$ $Q[\boldsymbol{u}]$,
			we infer from \eqref{VWUmleitung1} that \[ \mathbb{E}_{\boldsymbol{z}} \big[ f(\boldsymbol{Y}^{(n)}) \ \big{\vert} \ \Omega_{\mathrm{surv}}(n)  \big] = \frac{h(\boldsymbol{z})}{n Q[\boldsymbol{u}] \mathbb{P}_{\boldsymbol{z}}(\Omega_{\mathrm{surv}}(n))} \ \widehat{\mathbb{E}}_{\boldsymbol{z}} \bigg[ \frac{f(\boldsymbol{Y}^{(n)})}{h(\boldsymbol{Y}^{(n)}(1))}      \bigg].   \] Recalling from Definition~\ref{defhatP} that $h(\boldsymbol{x})= \langle \boldsymbol{u}, \boldsymbol{x} \rangle$, we already know from \eqref{eqlimitOmegasurvn} that \begin{equation*}
				n\mathbb{P}_{\boldsymbol{z}} (\Omega_{\mathrm{surv}}(n)) \underset{n \to \infty}{\longrightarrow} \frac{h( \boldsymbol{z})}{Q[\boldsymbol{u}]}. 
			\end{equation*} Thus, it remains to show that $\widehat{\mathbb{E}}_{\boldsymbol{z}} \big[ f(\boldsymbol{Y}^{(n)})/h(\boldsymbol{Y}^{(n)}(1)) \big] $ converges to $\mathbb{E}_{0} \big[  f(\boldsymbol{Y})/h(\boldsymbol{Y}(1))  \big]$. \\
			
			Fix some $\delta >0$. Then \begin{align}
				\widehat{\mathbb{E}}_{\boldsymbol{z}} \bigg[ \frac{f(\boldsymbol{Y}^{(n)})}{h(\boldsymbol{Y}^{(n)}(1))} \bigg] = \widehat{\mathbb{E}}_{\boldsymbol{z}} \bigg[ \frac{f(\boldsymbol{Y}^{(n)})}{h(\boldsymbol{Y}^{(n)}(1))} \mathds{1}_{\{ h(\boldsymbol{Y}^{(n)}(1)) \leq \delta \}} \bigg] + \widehat{\mathbb{E}}_{\boldsymbol{z}} \bigg[ \frac{f(\boldsymbol{Y}^{(n)})}{h(\boldsymbol{Y}^{(n)}(1))} \mathds{1}_{\{h(\boldsymbol{Y}^{(n)}(1)) > \delta    \}} \bigg]. \label{VWUmleitung2}
			\end{align} For the first expectation on the right-hand side, we use the assumption on $f$ and, once more, the fact that $h(\boldsymbol{Z}(n))= nh(\boldsymbol{Y}^{(n)}(1))$ to obtain 	
			\begin{align*}
				\widehat{\mathbb{E}}_{\boldsymbol{z}} \bigg[ \frac{f(\boldsymbol{Y}^{(n)})}{h(\boldsymbol{Y}^{(n)}(1))} \mathds{1}_{\{ h(\boldsymbol{Y}^{(n)}(1)) \leq \delta \}} \bigg] &\leq \widehat{\mathbb{E}}_{\boldsymbol{z}} \bigg[ \frac{1}{h(\boldsymbol{Y}^{(n)}(1))} \mathds{1}_{\{ h(\boldsymbol{Y}^{(n)}(1)) \leq \delta \}} \bigg] \\
				&= \frac{n}{h(\boldsymbol{z})} \ \mathbb{P}_{\boldsymbol{z}}(h(\boldsymbol{Y}^{(n)}(1)) \leq \delta, \ \Omega_{\mathrm{surv}}(n)) \\
				&= \frac{n \mathbb{P}_{\boldsymbol{z}} (\Omega_{\mathrm{surv}}(n)) }{h(\boldsymbol{z})} \ \mathbb{P}_{\boldsymbol{z}}(h(\boldsymbol{Y}^{(n)}(1)) \leq \delta \ \vert \ \Omega_{\mathrm{surv}}(n)).
			\end{align*} By \eqref{resjoffespitzer}, $\mathbb{P}_{\boldsymbol{z}}(h(\boldsymbol{Y}^{(n)}(1)) \leq \delta \ \vert \ \Omega_{\mathrm{surv}}(n)) \leq c \delta$ for all $n \geq 1$ and some constant $c>0$ (independent of $\delta$). This implies that \begin{equation}
				\widehat{\mathbb{E}}_{\boldsymbol{z}} \bigg[ \frac{f(\boldsymbol{Y}^{(n)})}{h(\boldsymbol{Y}^{(n)}(1))} \mathds{1}_{\{ h(\boldsymbol{Y}^{(n)}(1)) \leq \delta \}} \bigg] \leq \frac{n \mathbb{P}_{\boldsymbol{z}} (\Omega_{\mathrm{surv}}(n)) }{h(\boldsymbol{z})}  c \delta. \label{VWUmleitung3}
			\end{equation}
			To deal with the second expectation on the right-hand side of \eqref{VWUmleitung2}, we notice that the functional \[ f_{\delta}(\boldsymbol{y}) \coloneq \frac{f(\boldsymbol{y})}{h(\boldsymbol{y}(1))} \ \mathds{1}_{\{ h(\boldsymbol{y}(1)) > \delta  \}}   \] is bounded and its set of discontinuities is contained in $\{ \boldsymbol{y} : h(\boldsymbol{y}(1))= \delta \}$, which has zero probability with respect to the limiting process $\boldsymbol{Y}(t)$. Therefore, due to Theorem~\ref{theobessel}, 
			\begin{equation}
				\widehat{\mathbb{E}}_{\boldsymbol{z}} \bigg[ \frac{f(\boldsymbol{Y}^{(n)})}{h(\boldsymbol{Y}^{(n)}(1))} \mathds{1}_{\{h(\boldsymbol{Y}^{(n)}(1)) > \delta    \}} \bigg] \underset{n \to \infty}{\longrightarrow} \mathbb{E}_{0} \bigg[  \frac{f(\boldsymbol{Y})}{h(\boldsymbol{Y}(1))} \mathds{1}_{\{h(\boldsymbol{Y}(1)) > \delta    \}}    \bigg] \label{VWUmleitung4}.
			\end{equation} 	
			Now, taking $\limsup$ as $n \to \infty$ in \eqref{VWUmleitung2} we obtain that \begin{align*} \limsup_{n \to \infty} \widehat{\mathbb{E}}_{\boldsymbol{z}} \bigg[ \frac{f(\boldsymbol{Y}^{(n)})}{h(\boldsymbol{Y}^{(n)}(1))} \bigg] =& \ \limsup_{n \to \infty} \widehat{\mathbb{E}}_{\boldsymbol{z}} \bigg[ \frac{f(\boldsymbol{Y}^{(n)})}{h(\boldsymbol{Y}^{(n)}(1))} \mathds{1}_{\{ h(\boldsymbol{Y}^{(n)}(1)) \leq \delta \}} \bigg] \\				&\qquad \qquad \qquad+ \limsup_{n \to \infty} \widehat{\mathbb{E}}_{\boldsymbol{z}} \bigg[ \frac{f(\boldsymbol{Y}^{(n)})}{h(\boldsymbol{Y}^{(n)}(1))} \mathds{1}_{\{h(\boldsymbol{Y}^{(n)}(1)) > \delta    \}} \bigg].    
				\end{align*} 
				The second $\limsup$ is then already given by \eqref{VWUmleitung4}. For the first $\limsup$, note that the left-hand side does not depend on $\delta$, so we may take $\delta \to 0$ to find \[ \lim_{\delta \to 0} \limsup_{n \to \infty}  \widehat{\mathbb{E}}_{\boldsymbol{z}} \bigg[ \frac{f(\boldsymbol{Y}^{(n)})}{h(\boldsymbol{Y}^{(n)}(1))} \mathds{1}_{\{ h(\boldsymbol{Y}^{(n)}(1)) \leq \delta \}} \bigg] \leq \lim_{\delta \to 0} \limsup_{n \to \infty}   \frac{n \mathbb{P}_{\boldsymbol{z}} (\Omega_{\mathrm{surv}}(n)) }{h(\boldsymbol{z})}  c \delta =0 \] by \eqref{VWUmleitung3} and \eqref{eqlimitOmegasurvn}, and therefore  \[ \limsup_{n \to \infty}	\widehat{\mathbb{E}}_{\boldsymbol{z}} \bigg[ \frac{f(\boldsymbol{Y}^{(n)})}{h(\boldsymbol{Y}^{(n)}(1))} \bigg] =  \mathbb{E}_{0} \bigg[  \frac{f(\boldsymbol{Y})}{h(\boldsymbol{Y}(1))}   \bigg].   \] Repeating the same arguments for $\liminf$ proves the result.
		\end{proof}
		
		
		\subsection{Fdd convergence}\label{secfdd}
		During this section, we establish the convergence of all finite-dimensional distributions under both $\mathbb{P}$ and $\widehat{\mathbb{P}}$. To do so, we first extend the entrance law of \cite[Thm.~6]{joffe1967multitype} to $\widehat{\mathbb{P}}_{\boldsymbol{z}}$ (Proposition~\ref{propentrancelaw}) and then identify the limiting transition probabilities (Proposition~\ref{proptransitionprob}).
		\begin{proposition}[Entrance law]
			\label{propentrancelaw}
			Let $\boldsymbol{Z}(0)= \boldsymbol{z} \in \mathbb{N}_{0}^{d}\setminus \{ \boldsymbol{0}\}$, $t \in [0,T]$, $T>0$. Then, under $\mathbb{P}_{\boldsymbol{z}}$, we have
			\begin{equation}\label{eqentrancelawP}
				\frac{\boldsymbol{Z}(\lfloor nt \rfloor)}{n} \ \Big{\vert} \ \Omega_{\mathrm{surv}}(\lfloor nt \rfloor) \ \Rightarrow \ W_{t} \boldsymbol{v} \quad \text{as } \, n \to \infty, 
			\end{equation} where $W_{t} \coloneq tW \sim \expo(1/(Q[\boldsymbol{u}]t))$, and, under $\widehat{\mathbb{P}}_{\boldsymbol{z}}$ 
			\begin{equation}
				\frac{\boldsymbol{Z}(\lfloor nt \rfloor)}{n} \ \Rightarrow \ W_{t}^{*} \boldsymbol{v} \quad \text{as } \, n \to \infty, 
				\label{eqentrancelawhatP}
			\end{equation} where $W_{t}^{*}$ has the size-biased distribution of $W_{t}$, namely $W_{t}^{*} \sim \Gamma(2, 1/(Q[\boldsymbol{u}]t) )$.
		\end{proposition}
		\begin{proof}
			We begin with the limiting entrance law under $\mathbb{P}$. Since $\lfloor nt \rfloor/n \longrightarrow t$ as $n \to \infty$, \eqref{resjoffespitzer} entails that, under $ \mathbb{P}_{\boldsymbol{z}}$,
			\[ \frac{\boldsymbol{Z}(\lfloor nt \rfloor)}{n} \ \Big{\vert} \ \Omega_{\mathrm{surv}}(\lfloor nt \rfloor)\Rightarrow tW \boldsymbol{v} = W_{t}\boldsymbol{v}.  
			\]
			
			For the limiting entrance law under $\widehat{\mathbb{P}}$, \eqref{eqlimitOmegasurvn} tells us that 
			\begin{equation} \label{eqlimitOmegasurvt}
			n \mathbb{P}_{\boldsymbol{z}}(\Omega_{\mathrm{surv}}(\lfloor nt \rfloor)) = \frac{n}{\lfloor nt \rfloor} \lfloor nt \rfloor \mathbb{P}_{\boldsymbol{z}}(\Omega_{\mathrm{surv}}(\lfloor nt \rfloor))  \underset{n \to \infty}{\longrightarrow} \frac{\langle \boldsymbol{u}, \boldsymbol{z} \rangle}{Q[\boldsymbol{u}]t}.  
			\end{equation}
			
			Using the definition of $\widehat{\mathbb{P}}$, we therefore get for $f: \mathbb{R}_{+}^{d} \rightarrow \mathbb{R}$ such that $\boldsymbol{x} \mapsto \langle \boldsymbol{u}, \boldsymbol{x}\rangle f(\boldsymbol{x})$ is bounded and continuous:   \begin{align*}
				\widehat{\mathbb{E}}_{\boldsymbol{z}}\bigg[  f\Big( \frac{\boldsymbol{Z}(\lfloor nt \rfloor)}{n}    \Big)  \bigg] &= \frac{1}{\langle \boldsymbol{u}, \boldsymbol{z}\rangle} \mathbb{E}_{\boldsymbol{z}} \bigg[ \big\langle \boldsymbol{u}, \boldsymbol{Z}(\lfloor nt \rfloor)  \big\rangle \ f\Big( \frac{\boldsymbol{Z}(\lfloor nt \rfloor)}{n}    \Big)    \bigg] \\
				&= \frac{n \mathbb{P}_{\boldsymbol{z}}(\Omega_{\mathrm{surv}}(\lfloor nt \rfloor))}{\langle \boldsymbol{u}, \boldsymbol{z} \rangle} \mathbb{E}_{\boldsymbol{z}} \bigg[ \Big\langle \boldsymbol{u}, \frac{\boldsymbol{Z}(\lfloor nt \rfloor)}{n}  \Big\rangle f\Big( \frac{\boldsymbol{Z}(\lfloor nt \rfloor)}{n}    \Big) \ \bigg{\vert} \ \Omega_{\mathrm{surv}}(\lfloor nt \rfloor)   \bigg], 
			\end{align*}  Taking $n \to \infty$, using \eqref{eqlimitOmegasurvt} and \eqref{resjoffespitzer}, and recalling 
			$\langle \boldsymbol{u}, \boldsymbol{v} \rangle = 1$, this yields \[ \lim_{n \to \infty} 	\widehat{\mathbb{E}}_{\boldsymbol{z}}\bigg[  f\Big( \frac{\boldsymbol{Z}(\lfloor nt \rfloor)}{n}    \Big)  \bigg] = \frac{\langle \boldsymbol{u}, \boldsymbol{v}\rangle}{Q[\boldsymbol{u}]t} \ \mathbb{E} [ W_{t} f(W_{t} \boldsymbol{v}) ] = \frac{\mathbb{E} [ W_{t} f(W_{t} \boldsymbol{v}) ]}{Q[\boldsymbol{u}]t}.   \]
			
			Since, in particular, $\boldsymbol{x} \mapsto \langle \boldsymbol{u}, \boldsymbol{x}\rangle e^{- \langle \boldsymbol{\lambda}, \boldsymbol{x} \rangle}$, $\boldsymbol{x},\boldsymbol{\lambda} \in \mathbb{R}_+^{d}$, is bounded and continuous, this already implies that the stationary measure under $\widehat{\mathbb{P}}$ is obtained from the original one by  size-biasing $W_t$. In particular, this turns the density of $W_{t}$, namely $f(w)=  e^{- w/(Q[\boldsymbol{u}]t)}/(Q[\boldsymbol{u}]t)$, into \begin{equation} \hat{f}(w)= \frac{w f(w)}{\int_{0}^{\infty} \nu f(\nu) \mathrm{d} \nu} = \frac{1}{(Q[\boldsymbol{u}]t)^2}we^{- \frac{w}{Q[\boldsymbol{u}]t}},  \label{EB4}
			\end{equation} which is the density of $\Gamma(2, 1/(Q[\boldsymbol{u}]t))= \expo(1/(Q[\boldsymbol{u}]t)) \star \expo(1/(Q[\boldsymbol{u}]t)) $. This completes the proof. \qedhere	
		\end{proof}
		
		\begin{proposition}[Limiting transition probabilities]
			\label{proptransitionprob}
			Let $\boldsymbol{c} \in \mathbb{R}_{+}^{d}$, $t \in [0,T]$, $T>0$. Under $\mathbb{P}_{\lfloor n \boldsymbol{c} \rfloor}$, where the floor function of a vector is meant componentwise, we then have
			\begin{equation}
				\frac{\boldsymbol{Z}(\lfloor nt \rfloor)}{n} \Rightarrow  S_{K_{t},t} \boldsymbol{v}
				\quad \text{as } \, n \to \infty,
				 \label{eqtransitionP}
			\end{equation} where $S_{K_{t}, t}\coloneq \sum_{k =1}^{K_{t}} W_{t}^{(k)}$, $(W_{t}^{(k)})_{k \in \mathbb{N}}$ i.i.d. with $W_{t}^{(1)} \sim \expo\big(1/(Q[\boldsymbol{u}]t)\big)$, and $K_{t} \sim \poi\big(\langle \boldsymbol{u}, \boldsymbol{c} \rangle/(Q[\boldsymbol{u}]t)\big)$. Under $\widehat{\mathbb{P}}_{\lfloor n \boldsymbol{c} \rfloor}$, then
			\begin{equation}
				\frac{\boldsymbol{Z}(\lfloor nt \rfloor)}{n} 	\Rightarrow  (S_{K_{t},t})^{*} \boldsymbol{v} \quad \text{as } \, n \to \infty, \label{eqtransitionhatP}
			\end{equation} where $(S_{K_{t},t})^{*}$ has the size-biased distribution of $S_{K_{t}, t}$. In particular, for $\boldsymbol{c}=x \boldsymbol{v}$ with $x>0$, the density of $(S_{K_{t},t})^{*}$ reads \[ p_{t}(x,y) = e^{- \frac{x+y}{Q[\boldsymbol{u}]t}} \  \frac{1}{Q[\boldsymbol{u}]t} \ \frac{\sqrt{y}}{\sqrt{x}} \ \sum_{k=0}^{\infty} \frac{1}{k!(k+1)!} \ \bigg( \frac{\sqrt{x} \sqrt{y}}{Q[\boldsymbol{u}]t}  \bigg)^{2k+1},      \] where the right-hand side corresponds to the transition density of a $4$-dimensional squared Bessel process up to a scaling constant $Q[\boldsymbol{u}]$. 
		\end{proposition} 
		\begin{remark}\label{remmeaningKt}
			We will see that the random variable $K_{t}$ actually denotes the limit of the number of initial individuals that have at least one surviving descendant at time $\lfloor n t \rfloor$, on $\{ \boldsymbol{Z}(0)= \lfloor n \boldsymbol{c} \rfloor  \}$, as $n \to \infty$.
		\end{remark}
		
		\begin{proof}[Proof of Proposition \ref{proptransitionprob}]
			
			We begin with the limiting transition probabilities under $\mathbb{P}_{\lfloor n \boldsymbol{c} \rfloor}$ for $\boldsymbol{c} \in \mathbb{R}_{+}^{d}$. We may then write 
			\begin{equation}\label{Bernoulli}
			\frac{\boldsymbol{Z}(\lfloor nt \rfloor)}{n} \ \Big{\vert} \ \{ \boldsymbol{Z}(0)= \lfloor n \boldsymbol{c} \rfloor \} = \sum_{i=1}^{d} \sum_{j=1}^{\lfloor n c_{i} \rfloor }B^{(i,j)}( \lfloor nt \rfloor) \ \frac{\boldsymbol{Z}_{\mathrm{surv}}^{(i,j)}(\lfloor nt \rfloor)}{n},
			\end{equation}where, for every $j \in \mathbb{N}$, $\boldsymbol{Z}_{\mathrm{surv}}^{(i,j)}(\lfloor nt \rfloor)$ denotes an independent copy of $\boldsymbol{Z}(\lfloor nt \rfloor)$, started from $\boldsymbol{e}_{i}$ and conditioned on survival, while the $(B^{(i,j)} (\lfloor nt \rfloor))_{j \in \mathbb{N}}$ are i.i.d.\ Bernoulli random variables, independent of the $\boldsymbol{Z}_{\mathrm{surv}}^{(i,j)}(\lfloor nt \rfloor)$ and with parameter \begin{equation} p_{i, n,t} = \mathbb{P}_{\boldsymbol{e}_{i}}(\Omega_{\mathrm{surv}}(\lfloor nt \rfloor)) = \frac{u_{i}}{ Q[\boldsymbol{u}] \lfloor nt \rfloor} + \mathscr{o}\Big(\frac{1}{n}\Big). \label{eqdefpni} 
			\end{equation}In words, we first ask each initial individual whether it has surviving descendants at time $\lfloor nt \rfloor$. For those that do, we then apply \eqref{eqentrancelawP}, in the form \[ \frac{\boldsymbol{Z}_{\mathrm{surv}}^{(i,j)}(\lfloor nt \rfloor)}{n} \Rightarrow W_{t}^{(i,j)} \boldsymbol{v}  \quad \text{as } \, n \to \infty,  \] where the $(W_{t}^{(i,j)})_{i \in [d], j \in \mathbb{N}}$ denote independent copies of $W_{t}$, independently of $i$ and $j$. Since also \[  \lfloor n c_{i} \rfloor p_{i, n, t} = \frac{\lfloor n c_{i} \rfloor}{n} n p_{i, n, t}   \underset{n \to \infty}{\longrightarrow} \frac{u_{i}c_{i}}{Q[\boldsymbol{u}]t},  \] the independence of $\big( \boldsymbol{Z}_{\mathrm{surv}}^{(i,j)}(\lfloor nt \rfloor) \big)_{i \in [d],j \in \mathbb{N}}$ and  $\big(B^{(i,j)} (\lfloor nt \rfloor)\big)_{i, \in [d], j \in \mathbb{N}}$ yields the Poisson limit for every $i$ on the right-hand side of \eqref{Bernoulli}, so, under $\mathbb{P}_{\lfloor n \boldsymbol{c} \rfloor}$, \[  \frac{\boldsymbol{Z}(\lfloor nt \rfloor)}{n} \Rightarrow \sum_{i=1}^{d} \sum_{\ell =1}^{K_{t}^{(i)}} W_{t}^{(i, \ell)} \boldsymbol{v} \quad \text{as } \, n \to \infty,			\] where $K_{t}^{(i)} \sim \poi(u_{i}c_{i}/(Q[\boldsymbol{u}]t))$. In particular, for $K_{t} \sim \poi(\langle \boldsymbol{u}, \boldsymbol{c}\rangle/(Q[\boldsymbol{u}]t))$ and $(W_{t}^{(k)})_{k \in \mathbb{N}}$ a collection of independent copies of $W_{t}$, we get, under $\mathbb{P}_{\lfloor n \boldsymbol{c} \rfloor }$, \[  \frac{\boldsymbol{Z}(\lfloor nt \rfloor)}{n} 
			\Rightarrow\sum_{k =1}^{K_{t}} W_{t}^{(k)} \boldsymbol{v} = S_{K_{t},t}  \boldsymbol{v}, 
			\] which is statement \eqref{eqtransitionP} and provides the interpretation of $K_t$ announced in Remark \ref{remmeaningKt}. \\
			
			We now move on to the limiting transition probabilities under $\widehat{\mathbb{P}}$. Again, for $f: \mathbb{R}_{+}^{d} \rightarrow \mathbb{R}$ such that $\boldsymbol{x} \mapsto \langle \boldsymbol{u}, \boldsymbol{x}\rangle f(\boldsymbol{x})$ is bounded and continuous,  \begin{align*}
				\widehat{\mathbb{E}}_{\lfloor n\boldsymbol{c} \rfloor}\bigg[  f\Big(  \frac{\boldsymbol{Z}(\lfloor nt \rfloor)}{n} \Big)  \bigg] &= \frac{1}{\langle \boldsymbol{u}, \lfloor n\boldsymbol{c} \rfloor \rangle} \mathbb{E}_{\lfloor n\boldsymbol{c} \rfloor} \bigg[  \big\langle \boldsymbol{u}, \boldsymbol{Z}(\lfloor nt \rfloor) \big\rangle \ f\Big(  \frac{\boldsymbol{Z}(\lfloor nt \rfloor)}{n}   \Big)    \bigg] \\
				&=  \frac{n}{\langle \boldsymbol{u}, \lfloor n\boldsymbol{c} \rfloor\rangle} \mathbb{E}_{\lfloor n\boldsymbol{c} \rfloor} \bigg[ \Big\langle \boldsymbol{u}, \frac{\boldsymbol{Z}(\lfloor nt \rfloor)}{n} \Big\rangle \ f \Big(  \frac{\boldsymbol{Z}(\lfloor nt \rfloor)}{n} \Big)   \bigg],
			\end{align*} where, in the first step, we again used the definition of $\widehat{\mathbb{P}}$. Taking $n \to \infty$ this yields by \eqref{eqtransitionP} that 
			\begin{equation}\label{Ehatnc}
			\lim_{n \to \infty} \widehat{\mathbb{E}}_{\lfloor n\boldsymbol{c} \rfloor}\bigg[  f\Big(  \frac{\boldsymbol{Z}(\lfloor nt \rfloor)}{n} \Big)  \bigg]  = \frac{\langle \boldsymbol{u}, \boldsymbol{v} \rangle}{\langle \boldsymbol{u}, \boldsymbol{c} \rangle} \ \mathbb{E} [ S_{K_{t},t} f(S_{K_{t},t} \boldsymbol{v})   ] = \frac{ \mathbb{E} [ S_{K_{t},t} f(S_{K_{t},t} \boldsymbol{v}) ]}{\langle \boldsymbol{u}, \boldsymbol{c} \rangle},   
			\end{equation}
			where, in the last step, we again used $\langle \boldsymbol{u}, \boldsymbol{v} \rangle = 1$. On the other hand, due to Wald's identity, \begin{align*}
				\mathbb{E} [  S_{K_{t},t}  ]= \mathbb{E}(K_t) \mathbb{E}(W_t^{(1)}) =  \mathbb{E}[ K_{t}] Q[\boldsymbol{u}]t= \langle \boldsymbol{u}, \boldsymbol{c} \rangle.
			\end{align*} So again, since, in particular, $\boldsymbol{x} \mapsto \langle \boldsymbol{u}, \boldsymbol{x}\rangle e^{- \langle \boldsymbol{\lambda}, \boldsymbol{x} \rangle}$, $\boldsymbol{x}, \boldsymbol{\lambda} \in \mathbb{R}_+^d$, is bounded and continuous, this, together with \eqref{Ehatnc}, already implies that the stationary measure under $\widehat{\mathbb{P}}$ is the size-biased distribution of $S_{K_{t}, t}$ times $\boldsymbol{v}$. This proves \eqref{eqtransitionhatP}. \\
			
			Last, we determine the transition density of $(S_{K_{t},t})^{*}$. To this end, let $K_{t}^{*}$ and $(W_{t}^{(1)})^{*}$ denote random variables with the size-biased distribution of $K_{t}$ and $W_{t}^{(1)}$, respectively. Then according to \cite[Prop.~2.1]{arratia2019size}, we have \[ (S_{K_{t},t})^{*} = S_{K_{t}^{*}-1, t} + (W_{t}^{(1)})^{*} \quad \text{and} \quad K_{t}^{*} \overset{\text{d}}{=} K_{t}+1.  \]
			
			Since, due to \eqref{EB4}, $(W_{t}^{(1)})^{*} \overset{\text{d}}{=} W_{t}^{(k)} + W_{t}^{(j)}$ for arbitrary $k$ and $j$, this implies 
			\begin{equation}\label{SKtt}
			(S_{K_{t},t})^{*} \overset{\text{d}}{=} S_{K_{t}, t}+(W_{t}^{(1)})^{*}  \overset{\text{d}}{=} \sum_{k=1}^{K_{t}+2} W_{t}^{(k)}  \overset{\text{d}}{=} S_{K_{t}+2, t}.  
			\end{equation}
				Now, let $\boldsymbol{c}= x \boldsymbol{v}$	for some $x>0$. Then, by the definition of $K_t$, $K_{t} = K_{x,t} \sim \poi(x/(Q[\boldsymbol{u}]t))$,  and therefore
				\begin{equation} \pi_{x,k,t}\coloneq \mathbb{P}(K_{x,t}=k)= \Big( \frac{x}{ Q[\boldsymbol{u}]t}  \Big)^{k} \frac{1}{k!} e^{- \frac{x}{Q[\boldsymbol{u}]t}}. \label{eqpi}
				\end{equation} 
		        On the other hand, by the definition of $S_{K_{t},t}$, \begin{equation} S_{K_{x,t},t} \ \vert \ \{ K_{x,t}=k \} \sim \Gamma\Big(k,\frac{1}{Q[\boldsymbol{u}]t} \Big) \quad\text{with density} \quad  f_{k,t}(y) = \frac{1}{(k-1)! (Q[\boldsymbol{u}]t)^{k}} \ y^{k-1} \ e^{- \frac{y}{Q[\boldsymbol{u}]t}}. \label{eqfkt}
				\end{equation} Then, combining \eqref{SKtt}--\eqref{eqfkt} gives the density of $(S_{K_{t},t})^{*}$ as \begin{align*}
					p_{t}(x,y) &= \sum_{k=0}^{\infty} \pi_{x,k,t} f_{k+2,t}(y) = \sum_{k=0}^{\infty} \Big(  \frac{x}{Q[\boldsymbol{u}]t}  \Big)^{k} \ \frac{1}{k!} \ e^{- \frac{x}{Q[\boldsymbol{u}]t}} \cdot \frac{1}{(k+1)! (Q[\boldsymbol{u}]t)^{k+2}} \ y^{k+1} \ e^{- \frac{ y}{Q[\boldsymbol{u}]t}} \\
					&= e^{- \frac{x+y}{Q[\boldsymbol{u}]t}} \  \frac{1}{Q[\boldsymbol{u}]t} \ \frac{\sqrt{y}}{\sqrt{x}} \ \sum_{k=0}^{\infty} \frac{1}{k!(k+1)!} \ \bigg( \frac{\sqrt{x} \sqrt{y}}{Q[\boldsymbol{u}]t}  \bigg)^{2k+1}, 
				\end{align*} as claimed. \qedhere
			\end{proof}
			In particular, Propositions~\ref{propentrancelaw} and \ref{proptransitionprob} together yield
			\begin{corollary}\label{corfddconv}
				Let $\boldsymbol{Z}(0)= \boldsymbol{z} \in \mathbb{N}_{0}^{d}\setminus \{ \boldsymbol{0}\}$ and $T>0$. Then we have convergence of all finite-dimensional distributions of the processes $\big(\boldsymbol{Z}(\lfloor nt \rfloor)/(nQ[\boldsymbol{u}])\big)_{t \in [0,T]}$ and  $\big(\langle \boldsymbol{u}, \boldsymbol{Z}(\lfloor nt \rfloor)\rangle/(nQ[\boldsymbol{u}])\big)_{t \in [0,T]}$ to the finite-dimensional distributions of $\boldsymbol{Y}$ and $\langle \boldsymbol{u}, \boldsymbol{Y}\rangle$, respectively, under $\widehat{\mathbb{P}}_{\boldsymbol{z}}$. 
			\end{corollary}
			
			
			\subsection{Tightness}\label{sectightness}
			Throughout this section, let $\boldsymbol{Z}(0)= \boldsymbol{z}$ for some $\boldsymbol{z} \in \mathbb{N}_{0}^{d}\setminus \{ \boldsymbol{0}\}$ independent of $n$. Let $(\boldsymbol{w}_{1}, \ldots, \boldsymbol{w}_{d-1} )$ be linearly independent and orthogonal to $\boldsymbol{v}$. Since $\langle \boldsymbol{u},\boldsymbol{v}\rangle=1$, the set $\{\boldsymbol{u}, \boldsymbol{w}_{1}, \ldots, \boldsymbol{w}_{d-1}  \}$ is  linearly independent and forms a basis of $\mathbb{R}^{d}$. We shall decompose $\boldsymbol{Z}(\lfloor nt \rfloor)/n$ in this basis. It then suffices to prove that $( (\langle \boldsymbol{u}, \boldsymbol{Z}(\lfloor nt \rfloor)\rangle/n)_{t \in [0,T]}   )_{n \in \mathbb{N}_{0}}$ and $( (\langle \boldsymbol{w}_{\ell}, \boldsymbol{Z}(\lfloor nt \rfloor)\rangle/n)_{t \in [0,T]}   )_{n \in \mathbb{N}_{0}}$, $\ell \in [d-1]$, are tight. 
			
			\begin{proposition}[Tightness I]
				\label{proptightnessu}
				The sequence of processes \[ \bigg( \Big(\frac{\langle \boldsymbol{u}, \boldsymbol{Z}(\lfloor nt \rfloor)\rangle}{n} \Big)_{t \in [0,T]}   \bigg)_{n \in \mathbb{N}_{0}} \quad\text{is tight in $\mathcal{D}[0,T]$ under $\widehat{\mathbb{P}}$.}\] 
			\end{proposition}
			Before we turn to the proof, let us first establish two useful bounds:
			
			\begin{lemma}\label{lemVWFDD9}
				Let $H(n)\coloneq \langle \boldsymbol{u}, \boldsymbol{Z}(n) \rangle$. Then for some $C >0$, we have  \begin{equation*} \mathbb{E}\big[ (H(n)-H(0))^{2}   \big]  \leq CnH(0). 
				\end{equation*} 
			\end{lemma}
			\begin{proof}
				Since $(\boldsymbol{Z}(n))_{n \in \mathbb{N}_{0}}$ is time homogeneous, by the branching property and additivity of (conditionally) independent variances, it suffices to show that \begin{equation} \mathbb{E} \big[ (H(1)-H(0))^{2}  \big] \leq C H(0) \quad\text{for every }H(0).    \label{VWTA}
				\end{equation} Indeed, if \eqref{VWTA} is true,  \[ \mathbb{E} \big[ (H(n)-H(n-1))^{2}  \big]  = \mathbb{E}\Big[ \mathbb{E} \big[ (H(n)-H(n-1))^{2} \ \big{\vert} \  H(n-1)  \big]     \Big] \leq \mathbb{E}\big[ C H(n-1) \big] = CH(0),     \] where in the last step, we have used the martingale property of $(H(n))_{n \in \mathbb{N}_{0}}$. \\
				
				To show \eqref{VWTA}, we notice that
				\begin{equation}\label{EdiffH2}
					\mathbb{E} \big[ (H(1)-H(0))^{2}    \big] = \mathbb{V} [ H(1) ] = \mathbb{V} [ \langle \boldsymbol{u}, \boldsymbol{Z}(1) \rangle] \leq d \sum_{i=1}^{d} \mathbb{V}[  u_{i}Z_{i}(1) ] = d \sum_{i=1}^{d} u_{i}^{2} \mathbb{V} [  Z_{i}(1) ],
				\end{equation} and since every $Z_{i}(1)$ is a sum of independent random variables (see \eqref{eqdefgaltonwatson}), we have, for some $\widetilde{C}>0$, that \[ \mathbb{V}[  Z_{i}(1)] \leq \tilde{C} \sum_{j=1}^{d} Z_{j}(0)   \leq \frac{\widetilde{C}}{\min_{j} u_{j}} \langle \boldsymbol{u}, \boldsymbol{Z}(0) \rangle, \qquad i \in [d].   \] Together with \eqref{EdiffH2}, this implies that \[ \mathbb{E}\big[ (H(1)-H(0))^{2}  \big] \leq \frac{d \widetilde{C} \sum_{i=1}^{d}u_{j}^{2}}{\min_{j}u_{i}} \ H(0). \qedhere \]
				
			\end{proof}
			
			\begin{lemma}\label{lemVWthirdmoments}
				Assume that the third moments of the offspring distribution are finite. Then we have \[ \limsup_{n \to \infty} \frac{\mathbb{E}\big[ H(n)^{3}  \big]}{n^{2}} < \infty. \]
			\end{lemma}
			\begin{proof}
				To compute the third moment, we first notice that \begin{align*} \mathbb{E} \big[  H(n)^{3}   \big] =& \mathbb{E} \big[  H(n-1)^{3}  \big] + 3 \mathbb{E}\big[ H(n-1)^{2}(H(n)-H(n-1))  \big] \\ 
					&+ 3 \mathbb{E}\big[ H(n-1)(H(n)-H(n-1))^{2}   \big] + \mathbb{E}\big[ (H(n)-H(n-1))^{3}    \big]. 
				\end{align*} By the martingale property, \[ \mathbb{E} \big[ H(n-1)^{2}(H(n)-H(n-1))     \big] =0.  \] Moreover, using Lemma~\ref{lemVWFDD9} in the penultimate step,	we get for some $C_{1}, C_{2}>0$ \begin{align}
					\mathbb{E}\big[ H(n-1) (H(n)-H(n-1))^{2}  \big] &= \mathbb{E} \Big[ H(n-1) \mathbb{E}\big[ (H(n)-H(n-1))^{2} \ \big{\vert} \  H(n-1)    \big]    \Big] \nonumber \\
					&\leq  C \mathbb{E}\big[ H(n-1)^{2}  \big] \leq C_{1} \big[ (n-1)H(0)+H(0)^{2}   \big] \leq C_{2} (n-1). \label{eq2ZeileS10}
				\end{align} Therefore, for some $C_{3}>0$,
					\begin{align*}
						\mathbb{E} \big[ H(n)^{3}  \big] &\leq \mathbb{E} \big[ H(n-1)^{3}  \big] + 3 C_{2} (n-1) +  \mathbb{E}\big[ (H(n)-H(n-1))^{3}   \big].
					\end{align*} Iterating this inequality, we get
					\begin{align*}
						\mathbb{E} \big[ H(n)^{3}  \big]	&\leq H(0)^{3}+ C_{3}n^{2} + \sum_{k=1}^{n} \mathbb{E}\big[ (H(k)-H(k-1))^{3}  \big].
				\end{align*} Thus, it suffices to show that \begin{equation} \limsup_{n \to \infty} \frac{1}{n^{2}} \sum_{k=1}^{n} \mathbb{E}\big[ (H(k)-H(k-1))^{3}  \big] < \infty.  \label{eqlimdifH} 
				\end{equation}
				To estimate $\mathbb{E}\big[ (H(k)-H(k-1))^{3}  \big]$, one can again use  the fact that we have a sum of independent random variables with zero expectation. Explicitly, recalling \eqref{eqdefgaltonwatson},
				\begin{align*}
					H(1)= \sum_{i=1}^{d} u_{i}Z_{i}(1) = \sum_{i=1}^{d} u_{i} \sum_{\ell=1}^{d}  \sum_{j=1}^{Z_{\ell}(0)}  \xi_{\ell,i}^{(j)} = \sum_{\ell=1}^{d}  \sum_{j=1}^{Z_{\ell}(0)} \zeta_{j}^{(\ell)},
				\end{align*} where 
				\[
				 \zeta_{j}^{(\ell)} \coloneq  \sum_{i=1}^{d} u_{i}  \xi_{\ell,i}^{(j)};
				\] note that the upper index now refers to the types rather than the individuals, and the $(\zeta_{j}^{(\ell)})$ are independent. Since $H(0)= \mathbb{E}[ H(1)]$, this indeed yields \[ H(1)-H(0) = \sum_{i=1}^{d} \sum_{j=1}^{Z_{i}(0)} \Big(\zeta_{j}^{(i)} - \mathbb{E}\big[ \zeta_{j}^{(i)} \big] \Big),   \] so that, if we now set $\mathring{\zeta}_{j}^{(i)}= \zeta_{j}^{(i)} - \mathbb{E}\big[ \zeta_{j}^{(i)}  \big]$,  then for some $C, C' >0$ \begin{align*}
					\mathbb{E}\big[ (H(1)-H(0))^{3}  \big] &= \mathbb{E}\Bigg[ \bigg( \sum_{i=1}^{d}\sum_{j=1}^{Z_{i}(0)} \mathring{\zeta}_{j}^{(i)} \bigg)^{3}  \Bigg] =  \sum_{i=1}^{d} \sum_{j=1}^{Z_{i}(0)} \mathbb{E} \Big[ \big(\mathring{\zeta}_{j}^{(i)} \big)^{3}  \Big]   \leq C \sum_{i=1}^{d}Z_{i}(0) \leq C' H(0),
				\end{align*} where the penultimate step uses the third-moment assumption. Therefore, we have  \[ \limsup_{n \to \infty} \frac{1}{n^{2}} \sum_{k=1}^{n} \mathbb{E}\big[ (H(k)-H(k-1))^{3}  \big] \leq \limsup_{n \to \infty} \frac{C' H(0)}{n} < \infty,   \] from which the statement follows via \eqref{eqlimdifH}. \qedhere
			\end{proof}
			
			\begin{proof}[Proof of Proposition~\ref{proptightnessu}]
				Without loss of generality, we may assume $T=1$. Let $ \eta^{(n)}(t)\coloneq H(\lfloor nt \rfloor)/n = \langle \boldsymbol{u}, \boldsymbol{Z}(\lfloor nt \rfloor)\rangle/n$, $t \in [0,1]$. According to \cite[Thm.~13.5]{billingsley2013convergence}, it suffices to show that, for $r \leq s \leq t$ and $\lambda>0$, \begin{equation} \widehat{\mathbb{P}}_{\boldsymbol{z}}(|\eta^{(n)}(t)-\eta^{(n)}(s)| \geq \lambda, \ |\eta^{(n)}(s)-\eta^{(n)}(r)| \geq \lambda ) \leq \frac{(t-r)^{2\alpha}}{\lambda^{4\beta}} \label{eqtoshowtightness} 
				\end{equation} for some $\alpha > 1/2$ and $\beta >0$. Note that, if $t-r < 1/n$, either $t$ and $s$ or $s$ and $r$ fall into the same interval $[(k-1)/n, k/n )$ for some $k \in \mathbb{N}$, and the left hand side of \eqref{eqtoshowtightness} vanishes. Therefore, we may assume for the rest of the proof that $t-r \geq 1/n$. In this case, we have, in particular, \begin{equation}
					\frac{\lfloor nt \rfloor-\lfloor nr \rfloor}{n} \leq 2(t-r). \label{eqgaus}
				\end{equation}
				
				It is immediate from the definition of $\eta^{(n)}$ that \begin{align}
					\widehat{\mathbb{P}}_{\boldsymbol{z}} (|\eta^{(n)}(t)-&\eta^{(n)}(s)| \geq \lambda, \ |\eta^{(n)}(s)-\eta^{(n)}(r)| \geq \lambda )\nonumber \\
					&= \widehat{\mathbb{P}}_{\boldsymbol{z}}(|H(\lfloor nt \rfloor)-H(\lfloor ns \rfloor)| \geq \lambda n, \ |H(\lfloor ns \rfloor)-H(\lfloor nr \rfloor)| \geq \lambda n ) \nonumber \\
					&= \frac{1}{\langle \boldsymbol{u}, \boldsymbol{z}\rangle} \mathbb{E}_{\boldsymbol{z}} \big[ H(\lfloor nt \rfloor)  \mathds{1}_{\{|H(\lfloor nt \rfloor)-H(\lfloor ns \rfloor)| \geq \lambda n \} } \mathds{1}_{\{|H(\lfloor ns \rfloor)-H(\lfloor nr \rfloor)| \geq \lambda n \} }  \big]. \nonumber
				\end{align} We will now split this expectation into two parts $E_{1}$ and $E_{2}$ by adding and subtracting $H(\lfloor ns \rfloor)$, that is,
				\begin{align}
					\widehat{\mathbb{P}}_{\boldsymbol{z}} (|\eta^{(n)}(t)-&\eta^{(n)}(s)| \geq \lambda, \ |\eta^{(n)}(s)-\eta^{(n)}(r)| \geq \lambda )	=  \frac{1}{\langle \boldsymbol{u}, \boldsymbol{z}\rangle} (E_{1}+E_{2}), \label{VWFDD8}
				\end{align} where \begin{align*}
					E_{1}\coloneq& \mathbb{E}_{\boldsymbol{z}} \big[ (H(\lfloor nt \rfloor)-H(\lfloor ns \rfloor)) \mathds{1}_{\{|H(\lfloor nt \rfloor)-H(\lfloor ns \rfloor)| \geq \lambda n \} } \mathds{1}_{\{|H(\lfloor ns \rfloor)-H(\lfloor nr \rfloor)| \geq \lambda n \} }  \big] \\
					\text{and} \quad E_{2}\coloneq& \mathbb{E}_{\boldsymbol{z}} \big[ H(\lfloor ns \rfloor) \mathds{1}_{\{|H(\lfloor nt \rfloor)-H(\lfloor ns \rfloor)| \geq \lambda n \} } \mathds{1}_{\{|H(\lfloor ns \rfloor)-H(\lfloor nr \rfloor)| \geq \lambda n \} }  \big].
				\end{align*}
				
				For $E_{1}$, we first find that, by Markov's inequality, \begin{align*}
					E_{1} &\leq \frac{1}{\lambda n} \mathbb{E}_{\boldsymbol{z}} \big[ (H(\lfloor nt \rfloor)-H(\lfloor ns \rfloor))^{2} \mathds{1}_{\{|H(\lfloor ns \rfloor)-H(\lfloor nr \rfloor)| \geq \lambda n \} }  \big] \\
					&= \frac{1}{\lambda n} \mathbb{E}_{\boldsymbol{z}} \Big[\mathbb{E}_{\boldsymbol{z}} \big[ (H(\lfloor nt \rfloor)-H(\lfloor ns \rfloor))^{2}  \ \big{\vert} \   \mathcal{F}_{s} \big]  \mathds{1}_{\{|H(\lfloor ns \rfloor)-H(\lfloor nr \rfloor)| \geq \lambda n \} }  \Big],
				\end{align*} so that Lemma~\ref{lemVWFDD9} and \eqref{eqgaus} imply that for some $c_{1}>0$ \begin{align*}
					E_{1} \leq& \frac{2c_{1} (t-s)}{\lambda}  \mathbb{E}_{\boldsymbol{z}} \big[ H(\lfloor ns \rfloor) \mathds{1}_{\{|H(\lfloor ns \rfloor)-H(\lfloor nr \rfloor)| \geq \lambda n \} }  \big].
				\end{align*} Splitting this expectation again, this time by adding and subtracting $H(\lfloor nr \rfloor)$, gives
				\begin{align*}
					E_{1} \leq& \frac{2c_{1} (t-s)}{\lambda}  \mathbb{E}_{\boldsymbol{z}} \big[ (H(\lfloor ns \rfloor)-H(\lfloor nr \rfloor))  \mathds{1}_{\{|H(\lfloor ns \rfloor)-H(\lfloor nr \rfloor)| \geq \lambda n \} }  \big] \\
					&\qquad + \frac{2c_{1} (t-s)}{\lambda}  \mathbb{E}_{\boldsymbol{z}} \big[ H(\lfloor nr \rfloor)  \mathds{1}_{\{|H(\lfloor ns \rfloor)-H(\lfloor nr \rfloor)| \geq \lambda n \} }  \big] \\
					\leq& \frac{2c_{1} (t-s)}{\lambda^{2}n}  \mathbb{E}_{\boldsymbol{z}} \big[ (H(\lfloor ns \rfloor)-H(\lfloor nr \rfloor))^{2} \big] +\frac{2c_{1} (t-s)}{\lambda^{2}n}  \mathbb{E}_{\boldsymbol{z}} \big[ |H(\lfloor ns \rfloor)-H(\lfloor nr \rfloor)| \, H(\lfloor nr \rfloor)  \big].
				\end{align*} Using again Lemma~\ref{lemVWFDD9}, \eqref{eqgaus}, and, this time, also Jensen's inequality, we get for some $\widetilde{C}_{1}>0$ \begin{align*}
					E_{1} &\leq \frac{\widetilde{C}_{1}(t-s)(s-r)}{\lambda^{2}} \mathbb{E}_{\boldsymbol{z}} \big[ H(\lfloor nr \rfloor)  \big] + \frac{2c_{1}(t-s)}{\lambda^{2}n} \mathbb{E}_{\boldsymbol{z}} \Big[ \sqrt{\mathbb{E} \big[ (H(\lfloor ns \rfloor)-H(\lfloor nr \rfloor))^{2}  \ \big{\vert} \   \mathcal{F}_{r}   \big]}  H(\lfloor nr \rfloor)    \Big] \\
					&\leq \frac{\widetilde{C}_{1}(t-r)^{2}}{\lambda^{2}} \langle \boldsymbol{u}, \boldsymbol{z} \rangle + \frac{\widetilde{C}_{1} (t-r)^{\frac{3}{2}}}{\lambda^{2}}  \frac{1}{\sqrt{n}} \mathbb{E}_{\boldsymbol{z}} \big[ H(\lfloor nr \rfloor)^{\frac{3}{2}}    \big].
				\end{align*} Noting next that, for some $\widehat{C}_{1}, \bar{C}_{1}(\boldsymbol{z})>0$, \begin{align*}
					\frac{1}{\sqrt{n}} \mathbb{E}_{\boldsymbol{z}} \big[ H(\lfloor nr \rfloor)^{\frac{3}{2}}    \big] &= \frac{\mathbb{P}_{\boldsymbol{z}}(\boldsymbol{Z}(\lfloor nr \rfloor)>0)}{\sqrt{n}} \ (\lfloor nr \rfloor)^{\frac{3}{2}} \ \mathbb{E}_{\dg{\boldsymbol{z}}}\bigg[ \Big( \frac{H(\lfloor nr \rfloor)}{\lfloor nr \rfloor}  \Big)^{\frac{3}{2}} \  \bigg{\vert} \ \Omega_{\mathrm{surv}}(\lfloor nr \rfloor) \bigg] \\
					&\leq \widehat{C}_{1} r^{\frac{1}{2}} \boldsymbol{z} \bigg( \mathbb{E}_{\boldsymbol{z}} \bigg[ \Big( \frac{H(\lfloor nr \rfloor)}{\lfloor nr \rfloor}  \Big)^{2}  \ \bigg{\vert} \  \Omega_{\mathrm{surv}}(\lfloor nr \rfloor)   \bigg] \bigg)^{\frac{3}{4}} \leq \bar{C}_{1}(\boldsymbol{z}),
				\end{align*} we end up at  \begin{equation} E_{1} \leq C_{1}(\boldsymbol{z}) \frac{(t-r)^{\frac{3}{2}}}{\lambda^{2}} \quad \text{for some } C_1(\boldsymbol{z})>0. \label{VWFDD10}
				\end{equation}
				
				We now turn to $E_{2}$. Similarly as for $E_{1}$, we first find \begin{align*}
					E_{2} &\leq \frac{1}{\lambda^{2}n^{2}} \mathbb{E}_{\boldsymbol{z}} \big[ (H(\lfloor nt \rfloor)-H(\lfloor ns \rfloor))^{2} H(\lfloor ns \rfloor)  \mathds{1}_{\{ |H(\lfloor ns \rfloor)-H(\lfloor nr \rfloor)|\geq \lambda n \}}   \big] \\
					&= \frac{1}{\lambda^{2}n^{2}} \mathbb{E}_{\boldsymbol{z}} \Big[ \mathbb{E}_{\boldsymbol{z}} \big[(H(\lfloor nt \rfloor)-H(\lfloor ns \rfloor))^{2} \ \big{\vert} \  \mathcal{F}_{s} \big] H(\lfloor ns \rfloor)  \mathds{1}_{\{ |H(\lfloor ns \rfloor)-H(\lfloor nr \rfloor)|\geq \lambda n \}}   \Big],
				\end{align*} so that, again, by Lemma~\ref{lemVWFDD9} and \eqref{eqgaus}, we have for some $c_{2}>0$ \begin{align*}
					E_{2}&\leq \frac{c_{2}(t-s)}{\lambda^{2}n} \mathbb{E}_{\boldsymbol{z}} \big[ H(\lfloor ns \rfloor)^{2}  \mathds{1}_{\{ |H(\lfloor ns \rfloor)-H(\lfloor nr \rfloor)|\geq \lambda n \}}   \big].
				\end{align*} This time, splitting the expectation by adding and subtracting $H(\lfloor nr \rfloor)$ yields, for some $\widetilde{C}_{2}>0$,
				\begin{align*}
					E_{2} &\leq \frac{2c_{2}(t-s)}{\lambda^{2}n} \mathbb{E}_{\boldsymbol{z}} \big[ (H(\lfloor ns \rfloor)-H(\lfloor nr \rfloor))^{2} \big] + \frac{2c_{2}(t-s)}{\lambda^{2}n} \mathbb{E}_{\boldsymbol{z}} \big[ H(\lfloor nr \rfloor)^{2}  \mathds{1}_{\{ |H(\lfloor ns \rfloor)-H(\lfloor nr \rfloor)|\geq \lambda n \}}   \big] \\
					&\leq \frac{\widetilde{C}_{2}(t-s)(s-r)}{\lambda^{2}} + \frac{\widetilde{C}_{2}(t-s)}{\lambda^{2}n} \mathbb{E}_{\boldsymbol{z}} \big[ H(\lfloor nr \rfloor)^{2} \mathds{1}_{\{ |H(\lfloor ns \rfloor)-H(\lfloor nr \rfloor)|\geq \lambda n \}}   \big], 
				\end{align*} where, once more by Lemma~\ref{lemVWFDD9} and \eqref{eqgaus}, for some $\widehat{C}_{2}>0$, \begin{align*}
					\mathbb{E}_{\boldsymbol{z}} \big[ H(\lfloor nr \rfloor)^{2} \mathds{1}_{\{ |H(\lfloor ns \rfloor)-H(\lfloor nr \rfloor)|\geq \lambda n \}}   \big] &\leq \frac{1}{\lambda^{2}n^{2}} \mathbb{E}_{\boldsymbol{z}} \big[ H(\lfloor nr \rfloor)^{2} (H(\lfloor ns \rfloor)-H(\lfloor nr \rfloor))^{2} \big] \\
					&\leq \frac{\widehat{C}_{2}(s-r)}{\lambda^{2}n} \mathbb{E}_{\boldsymbol{z}}\big[ H(\lfloor nr \rfloor)^{3}  \big].
				\end{align*} But since $\mathbb{E}_{\boldsymbol{z}}\big[ H(\lfloor nr \rfloor)^{3}  \big]/n^{2}$ is bounded according to Lemma~\ref{lemVWthirdmoments}, we already find, for some $\tilde{c}_{2}, C_{2}(\boldsymbol{z}) >0$, that \begin{align}
					E_{2} \leq \frac{\tilde{c}_{2} (t-r)^{2}}{\lambda^{2}}+ \frac{\tilde{c}_{2} (t-r)^{2}}{\lambda^{4}} \frac{\mathbb{E}_{\boldsymbol{z}} \big[ H(\lfloor nr \rfloor)^{3}  \big]}{n^{2}} \leq C_{2}(\boldsymbol{z}) (t-r)^{2} \Big( \frac{1}{\lambda^{2}} + \frac{1}{\lambda^{4}}    \Big) \label{VWFDD11}.
				\end{align} Combining \eqref{VWFDD10} and \eqref{VWFDD11}, we finally conclude that, for some $C>0$, \[  \widehat{\mathbb{P}}_{\boldsymbol{z}}(|\eta^{(n)}(t)-\eta^{(n)}(s)| \geq \lambda, \ |\eta^{(n)}(s)-\eta^{(n)}(r)| \geq \lambda ) \leq \frac{C(t-r)^{\frac{3}{2}}}{\lambda^{4}}   \text{, for all } \lambda \in (0,1). \qedhere \]
			\end{proof}
			
			\begin{proposition}[Tightness II]
				\label{proptightnessw}
				For every $\ell \in [d-1]$, the sequence of processes \[ \bigg( \Big(\frac{\langle \boldsymbol{w}_{\ell}, \boldsymbol{Z}(\lfloor nt \rfloor)\rangle}{n} \Big)_{t \in [0,T]}   \bigg)_{n \in \mathbb{N}_{0}} \quad\text{is tight in $\mathcal{D}[0,T]$ under $\widehat{\mathbb{P}}$.} \]
			\end{proposition}

			Again, before we prove the result, we first need the following:
			
			\begin{lemma} \label{lemEwZ2}
				For every $\ell \in [d-1]$, we have \[ \lim_{n \to \infty} \mathbb{E}_{\boldsymbol{z}} \big[ \langle \boldsymbol{w}_{\ell}, \boldsymbol{Z}(n) \rangle^{2}   \big] < \infty.   \]
			\end{lemma}
			\begin{proof}	
				Let \[ \mathbf{D}_{\boldsymbol{z}} (n) \coloneq (\mathbb{E}_{\boldsymbol{z}} [Z_{i}(n)Z_{j}(n)] )_{i,j \in [d]} \quad \text{and} \quad  \mathbf{C}_{\boldsymbol{e}_{i}}( \boldsymbol{Z}(1)) = ( \mathbb{E}_{\boldsymbol{e}_{i}}[Z_{j}(1)Z_{k}(1)]- \mathbb{E}_{\boldsymbol{e}_{i}}[Z_{j}(1)]\mathbb{E}_{\boldsymbol{e}_{i}}[Z_{k}(1)]   )_{j,k \in [d]}.    \] Then, according to \cite[Chap.~II,~(4.3)]{harris1963theory},  we have \begin{equation} \mathbf{D}_{\boldsymbol{z}} (n)= (\mathbf{M}^{T})^{n} \mathbf{D}_{\boldsymbol{z}}(0) \mathbf{M}^{n} + \sum_{j=1}^{n} (\mathbf{M}^{T})^{n-j} \bigg( \sum_{i=1}^{d} \mathbf{C}_{\boldsymbol{e}_{i}}( \boldsymbol{Z}(1)) \mathbb{E}_{\boldsymbol{z}}[Z_{i}(j-1)]   \bigg) \mathbf{M}^{n-j}.  \label{SMDz(n)}   
				\end{equation}  When we also note that $\mathbb{E}_{\boldsymbol{z}}[\langle \boldsymbol{w}_{\ell}, \boldsymbol{Z}(n)\rangle^{2}] = \boldsymbol{w}_{\ell}^{T} \mathbf{D}_{\boldsymbol{z}}(n) \boldsymbol{w}_{\ell}$, \eqref{SMDz(n)} yields \begin{align} &\mathbb{E}_{\boldsymbol{z}}[\langle \boldsymbol{w}_{\ell}, \boldsymbol{Z}(n)\rangle^{2}] \nonumber \\
					&\qquad \qquad = (\mathbf{M}^{n} \boldsymbol{w}_{\ell})^{T} \mathbf{D}_{\boldsymbol{z}}(0) \mathbf{M}^{n} \boldsymbol{w}_{\ell} + \sum_{j=1}^{n} (\mathbf{M}^{n-j} \boldsymbol{w}_{\ell})^{T} \bigg( \sum_{i=1}^{d} \mathbf{C}_{\boldsymbol{e}_{i}}( \boldsymbol{Z}(1)) \mathbb{E}_{\boldsymbol{z}}[Z_{i}(j-1)]   \bigg) \mathbf{M}^{n-j} \boldsymbol{w}_{\ell}. \label{eqlem410}
				\end{align}   But according to Perron--Frobemus theory (see, for example, \cite[Thm.~II.5.1]{harris1963theory} or \cite[App.~2,~Thm.~2.3]{karlintaylor1975}), \[ \mathbf{M}^{n}= \mathbf{M}_{1}+ \mathbf{M}_{2}^{n}  \] with $\mathbf{M}_{1}= (u_{i}v_{j})_{i,j}$ and $\| \mathbf{M}_{2}^{n} \| \in \mathscr{O}(\alpha^{n})$, for some $\mathbf{M}_{2}$ and $\alpha$ with $0< \alpha <1$. In particular
				\[ \mathbf{M}_{1} \boldsymbol{w}_{\ell}= \sum_{i=1}^{d} u_{i} \langle \boldsymbol{v}, \boldsymbol{w}_{\ell} \rangle =0 \quad \text{and} \quad \| \mathbf{M}_{2}^{n} \boldsymbol{w}_{\ell} \| \leq \| \mathbf{M}_{2}^{n} \| \| \boldsymbol{w}_{\ell} \| = \mathscr{O}(\alpha^{n}).      \] Applying this to all terms in \eqref{eqlem410} proves the claim.
			\end{proof}

			\begin{proof}[Proof of Proposition~\ref{proptightnessw}]
				
				We may again assume $T=1$. The tightness of $((\langle \boldsymbol{w}_{\ell}, \boldsymbol{Z}(\lfloor nt \rfloor) \rangle/n)_{t \in [0,1]})_{n \in \mathbb{N}_{0}}$ will follow from the sufficient criterion \[ \widehat{\mathbb{P}}_{\boldsymbol{z}} \Big(  \max_{k \leq n} | \langle \boldsymbol{w}_{\ell}, \boldsymbol{Z}(k) \rangle | > \varepsilon n   \Big) \underset{n \to \infty}{\longrightarrow} 0 \qquad \text{for every } \varepsilon>0.  \] 
				
				Fix some $A>0$. By the definition of $\widehat{\mathbb{P}}$, \begin{align}
					\widehat{\mathbb{P}}_{\boldsymbol{z}} \Big(  \max_{k \leq n} | \langle \boldsymbol{w}_{\ell}, \boldsymbol{Z}(k) \rangle | > \varepsilon n   \Big)  =& \frac{1}{\langle \boldsymbol{u}, \boldsymbol{z} \rangle} \mathbb{E}_{\boldsymbol{z}} \big[  \langle \boldsymbol{u}, \boldsymbol{Z}(n) \rangle \mathds{1}_{\{ \max_{k \leq n} | \langle \boldsymbol{w}_{\ell}, \boldsymbol{Z}(k) \rangle | > \varepsilon n   \}}   \big] \nonumber \\
					\leq& \frac{An}{\langle \boldsymbol{u}, \boldsymbol{z} \rangle} \mathbb{P}_{\boldsymbol{z}} \Big( \max_{k \leq n} | \langle \boldsymbol{w}_{\ell}, \boldsymbol{Z}(k) \rangle | > \varepsilon n, \ \Omega_{\mathrm{surv}}(n)    \Big) \nonumber \\
					&+ \frac{1}{\langle \boldsymbol{u}, \boldsymbol{z} \rangle} \mathbb{E}_{\boldsymbol{z}} \big[ \langle \boldsymbol{u}, \boldsymbol{Z}(n) \rangle  \mathds{1}_{\{ \langle \boldsymbol{u}, \boldsymbol{Z}(n) \rangle > An  \}}    \big]. \label{VWT1}
				\end{align}
				
				By the Markov inequality and the second line in \eqref{eq2ZeileS10} (recalling $ \langle \boldsymbol{u}, \boldsymbol{Z}(n) \rangle = H(n)$), for some $C_{1}>0$ \begin{equation} \mathbb{E}_{\boldsymbol{z}}\big[ \langle \boldsymbol{u}, \boldsymbol{Z}(n) \rangle \mathds{1}_{\{ \langle \boldsymbol{u}, \boldsymbol{Z}(n) \rangle > An  \}}     \big] \leq \frac{1}{An} \mathbb{E}_{\boldsymbol{z}} \big[ \langle \boldsymbol{u}, \boldsymbol{Z}(n) \rangle^{2}   \big] \leq \frac{C_{1}n}{An} = \frac{C_{1}}{A}  \label{VWT2}
				\end{equation} (here and in what follows, all constants may depend on the starting state $\boldsymbol{z}$). 
				\\ 
				
				Fix also some $\delta \in (0,1)$. We now want to construct an upper bound for \[ \mathbb{P}_{\boldsymbol{z}} \Big(   \max_{k \leq \delta n} | \langle \boldsymbol{w}_{\ell}, \boldsymbol{Z}(k) \rangle | > \varepsilon n \Big) \] in terms of $\max_{k \leq \delta n}  \langle \boldsymbol{u}, \boldsymbol{Z}(k) \rangle$. In order to do so, we first recall that, due to Perron--Frobenius, both $\boldsymbol{u} > 0$ and $\boldsymbol{v} > 0$. On the one hand, this entails that there exists $\gamma_{1} >0$ such that \[ 0 < \gamma_{1} \| \boldsymbol{Z}(k) \| \leq \langle \boldsymbol{u}, \boldsymbol{Z}(k) \rangle \quad \text{for all } k \in \mathbb{N}_{0}.  \] On the other hand, since $\boldsymbol{w}_{\ell}$ is a basis vector, one has $\boldsymbol{w}_{\ell} \neq 0$, and so
				 $\langle \boldsymbol{w}_{\ell}, \boldsymbol{v} \rangle =0$ entails that $\boldsymbol{w}_{\ell}$ must have at least one strictly negative entry. Therefore, there is also a second constant $\gamma_{2}>0$ such that, on $\Omega_{\mathrm{surv}}(k)$, \[ 0 \leq | \langle \boldsymbol{w}_{\ell}, \boldsymbol{Z}(k) \rangle | \leq \gamma_{2} \| \boldsymbol{Z}(k) \| \quad \text{for all } k \in \mathbb{N}_{0}.    \] Consequently  \[ \frac{| \langle \boldsymbol{w}_{\ell}, \boldsymbol{Z}(k) \rangle |}{\langle \boldsymbol{u}, \boldsymbol{Z}(k) \rangle} \leq \frac{\gamma_{2}}{\gamma_{1}} =: \gamma,   \] so \[ | \langle \boldsymbol{w}_{\ell}, \boldsymbol{Z}(k) \rangle | \leq \gamma  \langle \boldsymbol{u}, \boldsymbol{Z}(k) \rangle  \quad \text{for all } k \in \mathbb{N}_{0}.    \] Therefore, for some $c(\boldsymbol{v}, \boldsymbol{u})>0$, 
				\[  \mathbb{P}_{\boldsymbol{z}} \Big(   \max_{k \leq \delta n} | \langle \boldsymbol{w}_{\ell}, \boldsymbol{Z}(k) \rangle | > \varepsilon n   \Big) \leq \mathbb{P}_{\boldsymbol{z}} \Big(  \max_{k \leq \delta n}  \langle \boldsymbol{u}, \boldsymbol{Z}(k) \rangle > c( \boldsymbol{v}, \boldsymbol{u}) \varepsilon n  \Big)  \leq \frac{\mathbb{E}_{\boldsymbol{z}} \big[  \max_{k \leq \delta n} \langle \boldsymbol{u}, \boldsymbol{Z}(k) \rangle^{2} \big]  }{ c (\boldsymbol{v}, \boldsymbol{u}) \varepsilon^{2}n^{2}}   .     \]

				Since $\langle \boldsymbol{u}, \boldsymbol{Z}(k) \rangle$ is a martingale,  Doob's $\mathcal{L}^{p}$ inequality tells us that, for some $C>0$,    \[ \mathbb{E}_{\boldsymbol{z}} \Big[  \max_{k \leq \delta n} \langle \boldsymbol{u}, \boldsymbol{Z}(k) \rangle^{2} \Big] \leq 4 \mathbb{E}_{\boldsymbol{z}} \big[ \langle \boldsymbol{u}, \boldsymbol{Z}(\delta n) \rangle^{2}    \big] \leq C \delta n,   \] where we also  used the second line in \eqref{eq2ZeileS10} for the last step. Consequently, for some $c_2(\boldsymbol{u},\boldsymbol{v})>0$, 
				\begin{equation}
					\mathbb{P}_{\boldsymbol{z}}\Big(  \max_{k \leq \delta n} | \langle \boldsymbol{w}_{\ell}, \boldsymbol{Z}(k) \rangle | > \varepsilon n  \Big) \leq \frac{c_{2}(\boldsymbol{u},\boldsymbol{v})\delta}{\varepsilon^{2}n}. \label{VWT3}
				\end{equation}
				
				Applying \eqref{VWT2} and \eqref{VWT3} to the corresponding terms on the right-hand side of \eqref{VWT1}, we get, for some $c_3(\boldsymbol{u},\boldsymbol{v})>0$, that \[ \widehat{\mathbb{P}}_{\boldsymbol{z}} \Big(  \max_{k \leq n} | \langle \boldsymbol{w}_{\ell}, \boldsymbol{Z}(k) \rangle | > \varepsilon n  \Big) \leq c_{3} \Big( \frac{1}{A}+\frac{\delta A}{\varepsilon^{2}}   \Big) + \frac{An}{\langle \boldsymbol{u}, \boldsymbol{z} \rangle} \mathbb{P}_{\boldsymbol{z}} \Big(  \max_{\delta n < k \leq n} | \langle \boldsymbol{w}_{\ell}, \boldsymbol{Z}(k) \rangle | > \varepsilon n, \ \Omega_{\mathrm{surv}}(n)  \Big).     \] Thus, it remains to show that \begin{equation} \lim_{n \to \infty} n \mathbb{P}_{\boldsymbol{z}} \Big(  \max_{\delta n < k \leq n} | \langle \boldsymbol{w}_{\ell}, \boldsymbol{Z}(k) \rangle | > \varepsilon n, \ \Omega_{\mathrm{surv}}(n)  \Big) = 0   \label{VWT4}  
				\end{equation} for all $\delta, A >0$. \\
				
				By Markov's inequality, for $k \leq n$, \begin{align*}
						\mathbb{P}_{\boldsymbol{z}}\big(  | \langle \boldsymbol{w}_{\ell}, \boldsymbol{Z}(k) \rangle | > \varepsilon n   \big) &\leq \frac{\mathbb{E}_{\boldsymbol{z}} \big[   \langle \boldsymbol{w}_{\ell}, \boldsymbol{Z}(k) \rangle^{2}  \mathds{1}_{\{  | \langle \boldsymbol{w}_{\ell}, \boldsymbol{Z}(k) \rangle | > \varepsilon n \}}      \big]  }{\varepsilon^{2}n^{2}} \\
						&\leq \frac{k}{\varepsilon^{2}n^{2}} \mathbb{E}_{\boldsymbol{z}}\bigg[ \Big(  \frac{\langle \boldsymbol{w}_{\ell}, \boldsymbol{Z}(k) \rangle}{\sqrt{k}}    \Big)^{2} \mathds{1}_{\{  | \frac{\langle \boldsymbol{w}_{\ell}, \boldsymbol{Z}(k) \rangle}{\sqrt{k}} | > \varepsilon \sqrt{k}  \}}  \mathds{1}_{ \Omega_{\mathrm{surv}}(k) }       \bigg] \\
						&\leq \frac{k \mathbb{P}_{\boldsymbol{z}}(\Omega_{\mathrm{surv}}(k))}{\varepsilon^{2}n^{2}}   \mathbb{E}_{\boldsymbol{z}}\bigg[ \Big(  \frac{\langle \boldsymbol{w}_{\ell}, \boldsymbol{Z}(k) \rangle}{\sqrt{k}}    \Big)^{2} \mathds{1}_{\{  | \frac{\langle \boldsymbol{w}_{\ell}, \boldsymbol{Z}(k) \rangle}{\sqrt{k}} | > \varepsilon \sqrt{k}  \}} \ \bigg{\vert} \  \Omega_{\mathrm{surv}}(k)      \bigg].
					\end{align*} According to \cite[Thm.~5.2]{athreya1974functionals}, \[\lim_{k \to \infty} \mathbb{P}_{\boldsymbol{z}}\Big( \frac{\langle \boldsymbol{w}_{\ell}, \boldsymbol{Z}(k) \rangle}{\sqrt{k}} \in \cdot \ \Big{\vert} \ \Omega_{\mathrm{surv}}(k) \Big) = \mu(\cdot),   \] where $\mu$ denotes a symmetric Laplace distribution. Since $\mathbb{E}\big[ \langle \boldsymbol{w}_{\ell}, \boldsymbol{Z}(k) \rangle^{2}   \big]$ converges (see Lemma~\ref{lemEwZ2}), we may conclude that \[ \mathbb{P}_{\boldsymbol{z}}\big(  | \langle \boldsymbol{w}_{\ell}, \boldsymbol{Z}(k) \rangle | > \varepsilon n   \big) = \mathscr{o}\Big(\frac{1}{n^{2}}\Big) \quad \text{as } \, n \to \infty.  \]
					
					Summing over $k \in [\delta n, n]$, we arrive at $\eqref{VWT4}$.  \qedhere
				\end{proof} \begin{proof}[Proof of Theorem~\ref{theobessel}] Combining Propositions~\ref{proptightnessu} and \ref{proptightnessw}, we get the tightness of $(( \boldsymbol{Z}(\lfloor nt \rfloor)/n )_{t \in [0,T]})_{n \in \mathbb{N}_{0}}$ in $\mathcal{D}([0,T], \mathbb{R}^{d})$ under $\widehat{\mathbb{P}}$. Together with Propositions~\ref{propentrancelaw} and \ref{proptransitionprob} we obtain Theorem~\ref{theobessel}. 
				\end{proof}
				
				
				\section{Population average of ancestral types}\label{secpopulationaverage}
				In this section, we are interested in the behaviour of the sequence of types along a typical branch of the population tree. It turns out that this  is completely determined by the left and right eigenvectors $\boldsymbol{v}$ and $\boldsymbol{u}$. A key role is played by the probability vector $\boldsymbol{\alpha}\coloneq (u_{i}v_{i})_{i \in [d]}$. As observed in \cite[Coro.~1]{jagers1992stabilities}, \cite[Prop.~1]{jagers1996asymptotic}, and \cite{georgii2003supercritical} for the supercritical case, this probability vector describes the distribution of ancestral types of an equilibrium population with type frequencies given by $\boldsymbol{v}$. The vector $\boldsymbol{\alpha}$ will therefore be called the \emph{ancestral (type) distribution}. In the supercritical case, the strong law of large numbers is at work once more, and one has an almost sure concentration result. We will now show  that, in the critical case, $\boldsymbol{\alpha}$ will play the same role, but now in the sense of the weak law of large numbers.\\
				
				\begin{figure}[h]
					\includegraphics[width=0.8\textwidth]{populationaverage.tikz}
					\caption{A familiy tree for $d=3$, where $A^{(1)}_{\textcolor{navyblue}{\text{blue}}}(5)$ is highlighted in \textbf{bold}.}
					\label{figpopulationaverage}
				\end{figure}  
				To begin, we consider a typical individual $x \in X_{n}$ alive at some large time $n$ and ask for the type $\sigma(x(n-m))$ of its ancestor $x(n-m)$ living at some earlier time $n-m$. We will see that $\sigma(x(n-m))$ is asymptotically distributed according to $\boldsymbol{\alpha}$. Specifically, let $m < n \in \mathbb{N}$ and \begin{equation}
					\boldsymbol{A}^{(m)}(n) \coloneq (A_{i}^{(m)}(n))_{i \in [d]} \quad \text{with} \quad	A_{i}^{(m)}(n) = \frac{\sum_{x \in X_{i}(n-m)} \langle \boldsymbol{1}, \boldsymbol{Z}(x,n) \rangle}{\langle \boldsymbol{1}, \boldsymbol{Z}(n) \rangle } \label{eqdefAnm}
				\end{equation} be the empirical ancestral type distribution at time $n-m$ taken over the population $X_{n}$ (see Figure \ref{figpopulationaverage}). Here, $X_{i}(n) = \{ x \in X(n): \sigma(x)=i   \}$ and $\boldsymbol{Z}(x,n)$ denotes the vector holding the number of descendants of the various types of $x$ alive in generation $n$. \\
				
				Since the next result is independent of the  value of $\boldsymbol{Z}(0)= \boldsymbol{z} \in \mathbb{N}_{0}^{d} \setminus \{ \boldsymbol{0}\}$, for the remainder of the section we may write $\mathbb{P}$ and $\widehat{\mathbb{P}}$, neglecting the subscript. 
				
				\begin{theorem}\label{theopopaverageancestraltypes}
					Let $\boldsymbol{A}^{(m)}(n)$ be as in \eqref{eqdefAnm}. Then \[ \lim_{m \to \infty} \lim_{n \to \infty} \boldsymbol{A}^{(m)}(n) 
					= \boldsymbol{\alpha} \quad \text{in $\widehat{\mathbb{P}}$-probability.}   \] 
				\end{theorem}
				\begin{proof} According to \cite[Thm.~II.5.1]{harris1963theory}, we have $\mathbb{E}_{\boldsymbol{e}_{i}}[ Z_{j}(m)] = (\mathbf{M}^{m})_{i,j} \underset{m \to \infty}{\longrightarrow} u_{i}v_{j}$, which immediately implies \[ \mathbb{E}_{\boldsymbol{e}_{i}}[ \langle \boldsymbol{1}, \boldsymbol{Z}(m) \rangle] = \sum_{j \in [d]} \mathbb{E}_{\boldsymbol{e}_{i}} [ Z_{j}(m) ] \underset{m \to \infty}{\longrightarrow}  \sum_{j \in [d]} u_{i}v_{j} = u_{i} \langle \boldsymbol{1}, \boldsymbol{v} \rangle = u_{i}.     \] 	
					
					Therefore, setting $\boldsymbol{\alpha}^{(m)}\coloneq (\alpha_{i}^{(m)})_{i \in [d]}$ with  $\alpha_{i}^{(m)} \coloneq v_{i} \cdot \mathbb{E}_{\boldsymbol{e}_{i}} [ \langle \boldsymbol{1}, \boldsymbol{Z}(m) \rangle ]   $, the theorem will follow from \[ \lim_{n \to \infty} \boldsymbol{A}^{(m)}(n) = \boldsymbol{\alpha}^{(m)}, \quad \text{in $\widehat{\mathbb{P}}$-probability.}   \]
					
					In order to see this, we further set \[ \boldsymbol{C}_{i, m}(n) \coloneq \frac{1}{Z_{i}(n)} \sum_{x \in X_{i}(n)} \boldsymbol{Z}(x, n+m) \quad \text{and} \quad \Pi_{i}(n)\coloneq \frac{Z_{i}(n)}{\langle \boldsymbol{1}, \boldsymbol{Z}(n) \rangle} \] to find \begin{equation}A_{i}^{(m)}(n+m) = \frac{Z_{i}(n) \langle \boldsymbol{1}, \boldsymbol{C}_{i,m}(n)\rangle}{\langle \boldsymbol{1}, \boldsymbol{Z} (n+m) \rangle} = \frac{\Pi_{i}(n)  \langle \boldsymbol{1}, \boldsymbol{C}_{i,m}(n)\rangle}{\sum_{k \in [d]}\Pi_{k}(n)  \langle \boldsymbol{1}, \boldsymbol{C}_{k,m}(n)\rangle},      \label{eqdecompA}
					\end{equation} since \[\sum_{k \in [d]} Z_{k}(n)  \langle \boldsymbol{1}, \boldsymbol{C}_{k,m}(n)\rangle = \sum_{k \in [d]} \sum_{x \in X_{k}(n)} \langle \boldsymbol{1}, \boldsymbol{Z}(x, n+m) \rangle = \langle \boldsymbol{1}, \boldsymbol{Z}(n+m) \rangle.  \]   
					
					On the one hand, Proposition~\ref{propentrancelaw} (with $t=1$) already yields that \begin{equation} \lim_{n \to \infty} \Pi_{i}(n) = v_{i} \quad \text{in $\widehat{\mathbb{P}}$-probability}. \label{ancestraldistribconvpi}
					\end{equation}
					
					On the other hand, we now want to prove that, for $\boldsymbol{a}^{(m)}\coloneq \mathbb{E}_{\boldsymbol{e}_{i}}[  \boldsymbol{Z}(m)]$,
					\[ \lim_{n \to \infty}  \boldsymbol{C}_{i, m}(n) = \boldsymbol{a}^{(m)} \quad \text{in $\widehat{\mathbb{P}}$-probability.}  \]  By the definition of $\widehat{\mathbb{P}}$, we have for all $\varepsilon >0$  \[  \widehat{\mathbb{P}} \Big( \big|   \boldsymbol{C}_{i, m}(n) -\boldsymbol{a}^{(m)}  \big| > \varepsilon \Big) = \frac{1}{\langle \boldsymbol{u}, \boldsymbol{z} \rangle} \mathbb{E}\big[ \langle \boldsymbol{u}, \boldsymbol{Z}(n+m) \rangle  \mathds{1}_{B(n,m, \varepsilon)} \big],  \] where $B(n,m, \varepsilon) \coloneq \big\{  |  \boldsymbol{C}_{i, m}(n)-\boldsymbol{a}^{(m)}  | > \varepsilon    \big\}$. Define also $D(n, m, R) \coloneq \big\{ \frac{ \langle \boldsymbol{u}, \boldsymbol{Z}(n+m) \rangle}{ \langle \boldsymbol{u}, \boldsymbol{Z}(n) \rangle} \leq R \big\}$ for some $R>1$. Then \begin{align}  \widehat{\mathbb{P}} \Big( \big|    \boldsymbol{C}_{i, m}(n)-\boldsymbol{a}^{(m)}  \big| > \varepsilon \Big)  \leq& \frac{R}{\langle \boldsymbol{u}, \boldsymbol{z} \rangle} \mathbb{E} \big[ \langle \boldsymbol{u}, \boldsymbol{Z}(n) \rangle \mathds{1}_{B(n,m, \varepsilon)}      \big] \nonumber \\
						&+  \frac{1}{\langle \boldsymbol{u}, \boldsymbol{z} \rangle} \mathbb{E} \big[ \langle \boldsymbol{u}, \boldsymbol{Z}(n+m) \rangle \mathds{1}_{D^{c}(n,m,R)}   \big]. \label{eqpopaverageBD}
					\end{align} Let us bound the second expectation on the right-hand side: \begin{align}
						\mathbb{E} \big[ \langle \boldsymbol{u}, \boldsymbol{Z}(n+m) \rangle  \mathds{1}_{D^{c}(n,m,R)}   \big] =& \mathbb{E} \big[ \langle  \boldsymbol{u}, \boldsymbol{Z} (n)  \rangle  \mathds{1}_{D^{c}(n,m,R)}   \big] + \mathbb{E} \big[ (\langle \boldsymbol{u}, \boldsymbol{Z}(n+m) \rangle - \langle \boldsymbol{u}, \boldsymbol{Z}(n) \rangle)  \mathds{1}_{D^{c}(n,m,R)}  \big] \nonumber \\
						\leq& \mathbb{E} \big[ \langle  \boldsymbol{u}, \boldsymbol{Z} (n)  \rangle  \mathds{1}_{D^{c}(n,m,R)}   \big] +E_{1}+E_{2}, \label{E1E2}
					\end{align} where \begin{align*}
						E_{1}\coloneq& \mathbb{E} \big[ (\langle \boldsymbol{u}, \boldsymbol{Z}(n+m) \rangle - \langle \boldsymbol{u}, \boldsymbol{Z}(n) \rangle) \mathds{1}_{\{ \langle \boldsymbol{u}, \boldsymbol{Z}(n) \rangle \leq 1  \}}   \big] \\
						\text{and} \quad E_{2}\coloneq& \mathbb{E} \big[ (\langle \boldsymbol{u}, \boldsymbol{Z}(n+m) \rangle - \langle \boldsymbol{u}, \boldsymbol{Z}(n) \rangle) \mathds{1}_{\{ \langle \boldsymbol{u}, \boldsymbol{Z}(n) \rangle > 1  \}} \mathds{1}_{D^{c}(n,m, R)}  \big].
					\end{align*}
					
					For $E_{1}$, we find by the martingale property of $\langle \boldsymbol{u}, \boldsymbol{Z}(n)\rangle$: 
					\[ E_{1} \leq \mathbb{E} \big[\langle \boldsymbol{u}, \boldsymbol{Z}(n+m) \rangle  \mathds{1}_{\{ \langle \boldsymbol{u}, \boldsymbol{Z}(n)\rangle \leq 1  \}}  \big] = \mathbb{E}\big[ \mathbb{E}[  \langle \boldsymbol{u}, \boldsymbol{Z}(n+m) \rangle \ \vert \ \mathcal{F}_{n} ]  \mathds{1}_{\{ \langle \boldsymbol{u}, \boldsymbol{Z}(n)\rangle \leq 1  \}}  \big] = \mathbb{E} \big[\langle \boldsymbol{u}, \boldsymbol{Z}(n) \rangle  \mathds{1}_{\{ \langle \boldsymbol{u}, \boldsymbol{Z}(n)\rangle \leq 1  \}}  \big] \] and, continuing by the definition of $\widehat{\mathbb{P}}$, \begin{equation} E_{1} \leq \mathbb{E} \big[ \langle \boldsymbol{u}, \boldsymbol{Z}(n) \rangle  \mathds{1}_{\{\langle \boldsymbol{u}, \boldsymbol{Z}(n) \rangle\} \leq 1}   \big] = \langle \boldsymbol{u}, \boldsymbol{z} \rangle \widehat{\mathbb{P}}(\langle \boldsymbol{u}, \boldsymbol{Z}(n) \rangle \leq 1  ) \underset{n \to \infty}{\longrightarrow} 0 \label{eqpopaverageE1} 
					\end{equation} since, under $\widehat{\mathbb{P}}$,  $\frac{\langle \boldsymbol{u}, \boldsymbol{Z}(n) \rangle}{n}$ converges to a non-degenerate distribution with no atom at $0$ (see Proposition \ref{propentrancelaw} with $t=1$). \\
					
					Concerning $E_{2}$, if $\frac{ \langle \boldsymbol{u}, \boldsymbol{Z}(n+m) \rangle}{ \langle \boldsymbol{u}, \boldsymbol{Z}(n) \rangle} > R$, then \[ \frac{\langle \boldsymbol{u}, \boldsymbol{Z}(n+m)\rangle - \langle \boldsymbol{u}, \boldsymbol{Z}(n)\rangle + \langle \boldsymbol{u}, \boldsymbol{Z}(n)\rangle}{\langle \boldsymbol{u}, \boldsymbol{Z}(n)\rangle} > R \] and, if further $\langle \boldsymbol{u}, \boldsymbol{Z}(n) \rangle >1$, also \[  \langle \boldsymbol{u}, \boldsymbol{Z}(n+m)\rangle - \langle \boldsymbol{u}, \boldsymbol{Z}(n)\rangle > \langle \boldsymbol{u}, \boldsymbol{Z}(n)\rangle (R-1).   \] So, by Markov's inequality, \[ E_{2} \leq \frac{1}{R-1} \mathbb{E} \big[  \langle \boldsymbol{u}, \boldsymbol{Z}(n+m)-\boldsymbol{Z}(n) \rangle^{2}    \big]. \]
					
					In particular, since the second moment of $\langle \boldsymbol{u}, \boldsymbol{Z}(n+m)-\boldsymbol{Z}(n) \rangle$ is finite (by Lemma~\ref{lemVWFDD9}), \begin{equation} E_{2} \leq \frac{Cm \langle \boldsymbol{u}, \boldsymbol{z} \rangle}{R-1} \label{eqpopaverageE2}
					\end{equation} for some $C>0$. Furthermore, the definition of $D(n, m, R)$ together with the martingale property imply that 
					\begin{equation}\label{EuZn} \mathbb{E}\big[ \langle \boldsymbol{u}, \boldsymbol{Z}(n) \rangle  \mathds{1}_{D^{c}(n,m, R)}     \big] \leq \frac{1}{R}  \mathbb{E}[\langle \boldsymbol{u}, \boldsymbol{Z}(n+m) \rangle] = \frac{\langle \boldsymbol{u}, \boldsymbol{z} \rangle}{R}.   
					\end{equation}
					
					Inserting \eqref{eqpopaverageE1}--\eqref{EuZn} into the right-hand side of \eqref{E1E2} and returning to \eqref{eqpopaverageBD}, we arrive at \begin{align} \widehat{\mathbb{P}} \Big( \big|\boldsymbol{C}_{i, m}(n) -\boldsymbol{a}^{(m)}  \big| > \varepsilon \Big)	\leq  \frac{R}{\langle \boldsymbol{u}, \boldsymbol{z} \rangle} &\mathbb{E} \big[ \langle \boldsymbol{u}, \boldsymbol{Z}(n) \rangle \mathds{1}_{B(n,m, \varepsilon)}    \big] + \frac{1}{R}+ \widehat{\mathbb{P}}(\langle \boldsymbol{u}, \boldsymbol{Z}(n) \rangle \leq 1  )  + \frac{Cm}{R-1} . \label{eqlln}
					\end{align}	
					
					Recalling that $Z_{i}(n)= | X_{i}(n)|$ and that the $((\boldsymbol{Z}(x, n+m))_{m \in \mathbb{N}_{0}})_{x \in X_{i}(n)}$ are i.i.d.\ critical multitype branching processes in their own right, all started with $\boldsymbol{e}_{i}$, we may write \[ \mathds{1}_{B(n,m, \varepsilon)} = \sum_{k=1}^{\infty} \mathds{1}_{B(n,m, \varepsilon)} \mathds{1}_{\{Z_{i}(n)=k \}} = \sum_{k=1}^{\infty} \mathds{1}_{\{ Z_{i}(n)=k  \}} \mathds{1}_{C_{k}},  \] where  $C_{k} \coloneq \big\{ | \big (\sum_{\ell=1}^{k} \boldsymbol{Z}^{(\ell)}(m) \big ) /k-\boldsymbol{a}^{(m)}   | > \varepsilon   \big\}$ and the  $((\boldsymbol{Z}^{(\ell)}(m))_{m \in \mathbb{N}})_{\ell \in \mathbb{N}}$ denote i.i.d. copies of our critical multitype branching process, all started with $\boldsymbol{e}_{i}$. In particular, due to their independence of $\boldsymbol{Z}(n)$, \[\mathbb{E}[\mathds{1}_{B(n,m, \varepsilon)} \ \vert \ \mathcal{F}_{n} ] = \sum_{k=1}^{\infty} \mathds{1}_{\{Z_{i}(n)=k  \}}  \mathbb{P}(C_{k}). \] Therefore, for $r >0$,
					\begin{equation} \begin{split}\label{uZn1} \mathbb{E} \big[ \langle \boldsymbol{u}, \boldsymbol{Z}(n) \rangle  \mathds{1}_{B(n,m, \varepsilon)}  \big]  =& \mathbb{E} \Big[ \langle \boldsymbol{u}, \boldsymbol{Z}(n) \rangle \mathbb{E}\big[ \mathds{1}_{B(n,m, \varepsilon)} \ \big\vert \ \mathcal{F}_{n} \big]  \Big] = \mathbb{E} \bigg[ \langle \boldsymbol{u}, \boldsymbol{Z}(n) \rangle  \sum_{k=1}^{\infty} \mathds{1}_{\{Z_{i}(n)=k  \}} \mathbb{P}(C_{k})   \bigg] \\
						\leq&  \mathbb{E} \big[ \langle \boldsymbol{u}, \boldsymbol{Z}(n) \rangle  \mathds{1}_{\{ Z_{i}(n) \leq r \}}  \big] + \max_{k>r} \ \mathbb{P}(C_{k}) \  \mathbb{E} [ \langle \boldsymbol{u}, \boldsymbol{Z}(n) \rangle ] \\
						=& \langle \boldsymbol{u}, \boldsymbol{Z}(0) \rangle \Big(  \widehat{\mathbb{P}}(Z_{i}(n) \leq r) + \max_{k>r} \mathbb{P}(C_{k})  \Big),
					\end{split}\end{equation}
					where we used the martingale property of $\langle \boldsymbol{u}, \boldsymbol{Z}(n)\rangle$ in the last step. Then, analogously to \eqref{eqpopaverageE1}, \[ \widehat{\mathbb{P}}(Z_{i}(n) \leq r) \underset{n \to \infty}{\longrightarrow} 0  \] and, by the weak law of large numbers, \[   \max_{k>r} \mathbb{P}(C_{k})  \underset{r \to \infty}{\longrightarrow} 0.  \]
					Therefore, if we first insert \eqref{uZn1} into \eqref{eqlln} and then let $n \to \infty$, $r \to \infty$, and  $R \to \infty$ (in this order), we finally obtain \begin{equation}\lim_{n \to \infty}  \boldsymbol{C}_{i, m}(n) = \mathbb{E}_{\boldsymbol{e}_{i}} [  \boldsymbol{Z}(m)] \quad \text{in $\widehat{\mathbb{P}}$-probability}. \label{ancestraldistribconvC}
					\end{equation}
					
					Due to Slutsky's lemma, \eqref{eqdecompA}, \eqref{ancestraldistribconvpi}, and \eqref{ancestraldistribconvC} together entail \[ A_{i}^{(m)}(n+m) \Rightarrow \alpha_{i}^{(m)} \quad \text{as } \, n \to \infty \] and, since  $\alpha_{i}^{(m)} = v_{i} \cdot \mathbb{E}_{\boldsymbol{e}_{i}}[ \langle \boldsymbol{1}, \boldsymbol{Z}(m) \rangle ]   $ is deterministic, actually	\[ \lim_{n \to \infty} A_{i}^{(m)}(n+m) = \alpha_{i}^{(m)} \quad \text{in $\widehat{\mathbb{P}}$-probability}.   \] Therefore, \[ \lim_{m \to \infty} \lim_{n \to \infty} A_{i}^{(m)}(n) = \lim_{m \to \infty} v_{i} \mathbb{E}_{\boldsymbol{e}_{i}}[ \langle \boldsymbol{1}, \boldsymbol{Z}(m) \rangle ] = u_{i}v_{i} = \alpha_{i} \quad \text{in $\widehat{\mathbb{P}}$-probability.}  \qedhere  \]
					
				\end{proof}
				
				From Theorem~\ref{theopopaverageancestraltypes}, we can now derive the announced weak analogue of the almost-sure convergence in the supercritical case \cite[Thm.~3.1]{georgii2003supercritical}:
				
				\begin{theorem} \label{coropopaverage}
					For every $\varepsilon>0$, we have \[ \lim_{m \to \infty} \lim_{n \to \infty} \mathbb{P} \Big( \big|\boldsymbol{A}^{(m)}(n) - \boldsymbol{\alpha} \big| > \varepsilon \ \Big\vert \ \Omega_{\mathrm{surv}}(n)  \Big) = 0.  \]
				\end{theorem}
				\begin{proof}
					By \eqref{eqdecompA} and Slutsky's lemma, it suffices to show that, for every $\varepsilon>0$, \begin{equation}\ \lim_{n \to \infty} \mathbb{P}\Big( \big| \Pi_{i}(n)-v_{i} \big| > \varepsilon \ \Big\vert \ \Omega_{\mathrm{surv}}(n)  \Big) = 0 \label{eqcondpiconv}
					\end{equation} and
					\begin{equation} \lim_{n \to \infty} \mathbb{P}\Big( \big| \boldsymbol{C}_{i,m}(n) -\boldsymbol{a}^{(m)}  \big| > \varepsilon \ \Big\vert \ \Omega_{\mathrm{surv}}(n+m)  \Big) = 0. \label{eqcondcimconv}
					\end{equation}
					The limit $m \to \infty$ then follows in the same way as in the proof of Theorem~\ref{theopopaverageancestraltypes}.
					Clearly, \eqref{eqcondpiconv} follows immediately  from \eqref{resjoffespitzer}. It remains to transfer the convergence of $\boldsymbol{C}_{i,m}(n)$ from $\widehat{\mathbb{P}}$ to $\mathbb{P}( \cdot \ \vert \ \Omega_{\mathrm{surv}}(n+m) )$. \\
					
					By the definition of $\widehat{\mathbb{P}}$ and \eqref{ancestraldistribconvC}, \begin{equation} \lim_{n \to \infty} \mathbb{E}\big[ \langle \boldsymbol{u}, \boldsymbol{Z}(n+m) \rangle  \mathds{1}_{\{| \boldsymbol{C}_{i,m}(n) - \boldsymbol{a}^{(m)} | > \varepsilon \} }    \big] =0. \label{VWcondconv1}
					\end{equation}
					Now, fix some $\delta >0$. Again, by \eqref{resjoffespitzer} \begin{equation}
						\lim_{n \to \infty}	\mathbb{P} \big(\langle \boldsymbol{u}, \boldsymbol{Z}(n+m) \rangle \leq \delta (n+m) \ \big\vert \ \Omega_{\mathrm{surv}}(n+m) \big) \leq c \delta \label{VWcondconv2}
					\end{equation} with some $c > 0$. Thus it remains to show that 
					\begin{equation}\label{Cm2am}
					\lim_{n \to \infty} \mathbb{P}\Big( \big| \boldsymbol{C}_{i.m}(n) - \boldsymbol{a}^{(m)} \big| > \varepsilon, \ \langle \boldsymbol{u}, \boldsymbol{Z}(n+m) \rangle > \delta (n+m) \ \Big\vert \ \Omega_{\mathrm{surv}}(n+m)  \Big)=0.   
					\end{equation}
					To this end, we notice that \eqref{VWcondconv1} implies that \[ \mathbb{P} \Big( \big| \boldsymbol{C}_{i,m}(n) - \boldsymbol{a}^{(m)} \big| > \varepsilon, \ \langle \boldsymbol{u}, \boldsymbol{Z}(n+m) \rangle > \delta n   \Big) \leq \frac{1}{\delta n} \mathbb{E} \big[ \langle \boldsymbol{u}, \boldsymbol{Z}(n+m) \rangle \cdot \mathds{1}_{\{ | \boldsymbol{C}_{i,m}(n) - \boldsymbol{a}^{(m)} | > \varepsilon   \}} \big] = \mathscr{o}\Big(\frac{1}{n}\Big).  \] Recalling now that $\lim_{n \to \infty} (n+m) \mathbb{P}( \Omega_{\mathrm{surv}}(n+m))$ is positive and finite, we see that  \eqref{Cm2am} is indeed true. Combining this with \eqref{VWcondconv2} and taking $\delta \to 0$ after $n$, we obtain \eqref{eqcondcimconv}. \qedhere
				\end{proof}
				
				
				\section{Construction of the size-biased family tree}\label{secsizebiasing}
			
			We now construct the size-biased multitype Galton-Watson tree as introduced in \cite{kurtz1997conceptual} (see also \cite{georgii2003supercritical} for the continuous-time version). Informally, this is a tree with a randomly selected trunk (or spine) along which offspring is weighted according to its size; in particular, there is always at least one offspring along the trunk so that the trunk survives forever. The children off the trunk get ordinary (unbiased) descendant trees (the bushes). As anticipated in Section~\ref{secP}, this construction is connected to $\widehat{\mathbb{P}}$. Namely, taking $\boldsymbol{Z}(0)$=$\boldsymbol{e}_{i}$ for a given type $i \in [d]$, $\widehat{\mathbb{P}}_{\boldsymbol{e}_{i}}$ gives rise to the size-biased offspring distribution, that is, the offspring distribution of an $i$-individual on the trunk \begin{equation}   \widehat{p}_{i}(\boldsymbol{\kappa}) = \frac{\langle \boldsymbol{u}, \boldsymbol{\kappa} \rangle p_{i}(\boldsymbol{\kappa})}{u_{i}}, \quad \kappa \in \mathbb{N}_{0}^{d}. \label{HOG4.1}
				\end{equation}   We denote the offspring of such an individual as $\widehat{\boldsymbol{N}}_{i}= (\widehat{N}_{i,j})_{j \in [d]}$. 
				Only one of these children is chosen as the successor on the trunk with probability proportional to $u_{j}$ when their type is $j$.
			
				That is, the successor is of type $j$ with probability $\widehat{N}_{i,j}u_{j}/\langle \widehat{\boldsymbol{N}}_{i}, \boldsymbol{u} \rangle$ for a given offspring, and with probability \begin{equation}  p_{i,j} = \mathbb{E} \bigg[ \frac{\widehat{N}_{i,j}u_{j}}{\langle \widehat{\boldsymbol{N}}_{i}, \boldsymbol{u} \rangle}   \bigg] = \frac{m_{i,j}u_{j}}{u_{i}} \label{eqjumpprobabilityretrospectivemutationchain}
				\end{equation} on average. 	\\
				
				Let $\{ \boldsymbol{N}_{x} : x \in \mathbb{X}   \}$ be as in Section~\ref{secfamilytree} and, independently of this, a sequence $\{ \widehat{\boldsymbol{N}}_{n}, \xi_{n} : n \geq 0  \}$ of random variables with values in $\mathbb{N}_{0}^{d}$ and $\mathbb{X}$, respectively, such that, for a given type $i_{0}=i$ of the root, $\xi_{0}=i$ and \begin{itemize}
					\item $\widehat{\boldsymbol{N}}_{0}$ has distribution $\widehat{p}_{i}$ and $\xi_{1}$ has conditional distribution \[ \mathbb{P}(\xi_{1}=(i_{1}, \ell_{1}) \ \vert \ \widehat{\boldsymbol{N}}_{0} ) = \frac{u_{i_{1}}}{\langle \widehat{\boldsymbol{N}}_{0}, \boldsymbol{u} \rangle} \mathds{1}_{\{\ell_{1} \leq \widehat{N}_{0, i_{1}} \}}   \] for all $(i_{1}, \ell_{1})\in \mathbb{X}_{1}$. Here, $\xi_{1}$ denotes the first successor on the trunk.
					\item For any $n \geq 1$, conditionally on $\mathbb{F}_{n-1}= \sigma\{ \widehat{\boldsymbol{N}}_{k}, \xi_{k+1}  : k< n  \}$, $\widehat{\boldsymbol{N}}_{n}$ follows the law $\widehat{p}_{\sigma(\xi_{n})}$ and \[ \mathbb{P}(\xi_{n+1}= (\xi_{n}; i_{n+1}, \ell_{n+1})| \mathbb{F}_{n-1}, \widehat{\boldsymbol{N}}_{n}) = \frac{u_{i_{n+1}}}{\langle \widehat{\boldsymbol{N}}_{n}, \boldsymbol{u} \rangle} \mathds{1}_{\{ \ell_{n+1} \leq \widehat{N}_{n, i_{n+1}}  \}}   \] for all $(i_{n+1}, \ell_{n+1}) \in [d] \times \mathbb{N}$, i.e. $\xi_{n+1}$ is a child of $\xi_{n}$ selected randomly with weight proportional to $u_{\sigma(\xi_{n+1})}$.\\
				\end{itemize} 
				
				Define $\widehat{X} = \bigcup_{n \geq 0} \widehat{X}_{n} \subset \mathbb{X}$ recursively by $\widehat{X}_{0} = \{i\}$ and $\widehat{X}_{n}= \widehat{X}_{n}^{\sharp} \cup \widehat{X}_{n}^{\flat}$ with \[ \widehat{X}_{n}^{\sharp} = \{ (\xi_{n-1}; i_{n}, \ell_{n}) \in \mathbb{X}_{n} : \ell_{n} \leq \widehat{N}_{n-1, i_{n}}  \},   \] the offspring of $\xi_{n-1}$ and \[ \widehat{X}_{n}^{\flat} = \{ (\tilde{x}; i_{n}, \ell_{n}) \in \mathbb{X}_{n} : \tilde{x} \in \widehat{X}_{n-1} \setminus \{\xi_{n-1} \}, \ell_{n} \leq N_{\tilde{x}, i_{n}} \}   \] the offspring of all other individuals in $\widehat{X}_{n-1}$. (Note that, in the last display, there is no hat on $N$; that is, the bushes have unbiased offspring.) In particular, the type process along the trunk, $\sigma(n)\coloneq \sigma(\xi_{n})$, also referred to as \emph{retrospective mutation chain} in \cite{georgii2003supercritical}, has the jump probabilities displayed in \eqref{eqjumpprobabilityretrospectivemutationchain}:
				
				\begin{definition}\label{defretrospectivemutationchain}
					The \emph{retrospective mutation chain} is the Markov chain $(\sigma(n))_{n \in \mathbb{N}_{0}}$ on $[d]$ with transition matrix  $\mathbf{P}\coloneq (p_{i,j})_{i,j \in [d]}$, with the $p_{i,j}$ of \eqref{eqjumpprobabilityretrospectivemutationchain}. 
				\end{definition} In particular, $\boldsymbol{\alpha}= (u_{i}v_{i})_{i \in [d]}$ is  the (unique) stationary distribution of the retrospective mutation chain, since \begin{align*}
				    \sum_{i=1}^{d} \alpha_{i} p_{i,j} = \sum_{i=1}^{d} u_{i}v_{i} \frac{m_{i,j}u_{j} }{u_{i}} = v_{j}u_{j}=\alpha_{j}.
				\end{align*}

				\begin{figure}[h]
		        \includegraphics[width=0.7\textwidth]{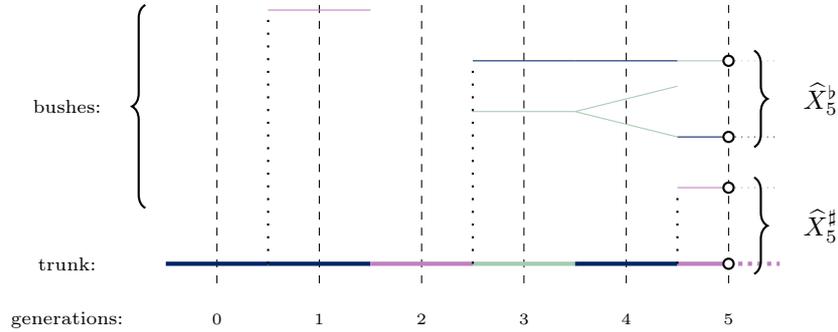}
	        	\caption{A size-biased Galton-Watson tree for $d=3$. The type process on the trunk corresponds to the retrospective mutation chain, while the collection of bushes (though still descendants of the trunk, as indicated by the dotted lines) may be understood as branching processes with immigration depending on the current state of the trunk.}
	        	\label{figtrunk}
        	\end{figure}

				In contrast, the collection of all individuals off the trunk may be understood as a branching process with immigration depending on the current type on the trunk, see Remark~\ref{DankaPap} (and Figure~\ref{figtrunk}). \\
				
				We write $\widehat{\mathbb{P}}_{\boldsymbol{e}_{i}}^{*}$ for the distribution of $(\widehat{X}_{n}, \xi_{n})_{n \in \mathbb{N}_{0}}$. The representation theorem below establishes the relationship between $\mathbb{P}_{\boldsymbol{e}_{i}}$, $\widehat{\mathbb{P}}_{\boldsymbol{e}_{i}}^{*}$, and the retrospective mutation chain. We use the shorthand $y[0,n]$ for a path $(y(m))_{m \in \{0, \dots, n\}}$.  
				\begin{proposition}[Many-to-one formula]\label{HOGtheo4.1}
					Let $n \in \mathbb{N}$, $i \in [d]$, and $F: A([n]_{0}, \mathfrak{B}_{f}(\mathbb{X}) \times \mathbb{X}) \mapsto [0,\infty)$ measurable, where $\mathfrak{B}_{f}(\mathbb{X})$ denotes the (countable) set of all finite subsets of $\mathbb{X}$ and $A([n]_{0}, \mathfrak{B}_{f}(\mathbb{X}) \times \mathbb{X})$ the space of all sequences on $[n]_{0}=\{0, \ldots, n\}$ taking values in $ \mathfrak{B}_{f}(\mathbb{X}) \times \mathbb{X}$. Then one has \begin{equation} u_{i}^{-1} \mathbb{E}_{\boldsymbol{e}_{i}} \bigg[  \sum_{x \in X_{n}} F(X[0,n], x[0,n]) u_{\sigma(x)}    \bigg] = \widehat{\mathbb{E}}_{\boldsymbol{e}_{i}}^{*} \big[ F(\widehat{X}[0,n], \xi[0,n])   \big]. \label{HOG4.2}
					\end{equation}
				\end{proposition}
				\begin{proof} Analogously to \cite{georgii2003supercritical} (except that we work in discrete time, so do not have a time change), it is sufficient to show that \begin{equation}
						\widehat{\mathbb{E}}_{\boldsymbol{e}_{i}}^{*} \big[ F(\widehat{X}[0,n], \xi[0,n]) \ | \ \xi_{n}=x   \big] = u_{i}^{-1} u_{\sigma(x)}  \mathbb{E}_{\boldsymbol{e}_{i}} \big[ F(X[0,n], x[0,n]) \ | \ x \in X_{n}    \big] \label{HOG5.2}
					\end{equation} for all $x \in X_{n}$; the theorem then follows by summation over all $x \in X_{n}$. Suppose that $x=(i_{1}, \ell_{1}; \ldots; i_{n}, \ell_{n}) \in \mathbb{X}_{n}$, and let $x_{k}=(i_{1}, \ell_{1}; \ldots: i_{k}, \ell_{k})$ be its ancestor in generation $k$, $0 \leq k \leq n$. Consider the right-hand side of $\eqref{HOG5.2}$ and write $u_{i}^{-1}u_{\sigma(x)}= q_{1}q_{2}$ with \[ q_{1}= \prod_{k=0}^{n-1} \frac{\langle \boldsymbol{N}_{x_{k}}, \boldsymbol{u} \rangle}{u_{i_{k}}}, \quad q_{2}= \prod_{k=0}^{n-1} \frac{u_{i_{k+1}}}{\langle \boldsymbol{N}_{x_{k}}, \boldsymbol{u} \rangle}  \]since $u_{i_{n}}= u_{\sigma(x)}$ and $i_{0}=i$. Of course, the random quantities $q_{1}$ and $q_{2}$ must then be included into the expectation.  \\
					
					First, it is immediate from \eqref{HOG4.1} that the factor $q_{1}$ is precisely the Radon-Nikodym density corresponding to a change from $N_{x_{k}}$ to the size-biased offspring $\widehat{N}_{k}$ for $k \in \{0, \ldots, n-1 \}$. Then, $q_{2}$ is the conditional selection probability for the trunk: \[ q_{2}= \widehat{\mathbb{P}}_{\boldsymbol{e}_{i}}^{*}(\xi_{k+1}=x_{k+1} \text{ for } 0 \leq k <n \ | \ \widehat{X}[0,n] ).  \] The right-hand side of \eqref{HOG5.2} is therefore equal to \[ 	\widehat{\mathbb{E}}_{\boldsymbol{e}_{i}}^{*} \big[ F(\widehat{X}[0,n], x[0,n]) \cdot \mathds{1}_{\{ \xi_{k+1}=x_{k+1} \text{ for } 0 \leq k <n  \}}   \big] = 	\widehat{\mathbb{E}}_{\boldsymbol{e}_{i}}^{*} \big[ F(\widehat{X}[0,n], \xi[0,n]) \ | \ \xi_{n}=x   \big]. \qedhere  \]
				\end{proof} 
				
				\begin{remark}\label{remstationarydistribretrospectivemutationchain}
					Since $\boldsymbol{\alpha}= (u_{i}v_{i})_{i \in [d]}$ is the (unique) stationary distribution of the retrospective mutation chain and $Z_{j}(n)= \sum_{x \in X_{n}} \mathds{1}_{\{\sigma(x)=j\}}$, we may use Proposition~\ref{HOGtheo4.1} to confirm \[ \mathbb{E}_{\boldsymbol{e}_{i}}[ Z_{j}(n)] = u_{i} \ \widehat{\mathbb{P}}_{\boldsymbol{e}_{i}}^{*} (\sigma(\xi(n))=j) \ u_{j}^{-1} \underset{n \to \infty}{\longrightarrow} u_{i}\alpha_{j}u_{j}^{-1} = u_{i}v_{j}.    \]
				\end{remark}
				
				
				\section{Time average of ancestral types}\label{sectimeaverage}
				
				In our next theorem, we ask for the time average of types along the line of descent leading to a typical $x \in X_{n}$. This time average is given by the empirical distribution  \begin{equation}
					L^{x}(n) = \frac{1}{n} \sum_{k=0}^{n-1} \delta_{\sigma(x(k))} \in \mathcal{P}([d]), \label{eqdefLxn}
				\end{equation} where $\mathcal{P}([d])$ denotes the space of probability measures on $[d]$ equipped with the weak topology, and $\sigma(x(k))$ gives us the type of $x$'s ancestor at time $k$. 
				
				\begin{theorem}\label{theolargedevtimeaverage}
				For every closed $F \subset \mathcal{P}([d])$ with $\boldsymbol{\alpha} \notin F$, one has \[\limsup_{n \to \infty} \frac{1}{n} \log \Big[\mathbb{P}_{\boldsymbol{e}_{i}}\Big( \sum_{x \in X_{n}} \mathds{1}_{\{L^{x}(n) \in F\}} > 0 \ \Big{\vert} \ \Omega_{\mathrm{surv}}(n) \Big)\Big] \leq  - \inf_{\boldsymbol{\nu} \in F}J_{\mathbf{P}}(\boldsymbol{\nu})  \] with the large-deviation rate function \[ J_{\mathbf{P}}(\boldsymbol{\nu}) \coloneq \sup_{\boldsymbol{w}>0}\bigg\{ - \sum_{j=1}^{d} \nu_{j} \log \Big[ \frac{(\mathbf{P} \boldsymbol{w})_{j}}{w_{j}}  \Big] \bigg\}  \] with $\mathbf{P}$ as in Definition~\ref{defretrospectivemutationchain}.
				\end{theorem}
				
				\begin{remark}\label{theotimeaverageancestraltypes}
				    In particular, this implies a weaker analogue of the strong law of large numbers in the supercritical case \cite[Thm.~3.2]{georgii2003supercritical}, namely: for every closed $F \subset \mathcal{P}([d])$ with $\boldsymbol{\alpha} \notin F$ and every $\varepsilon>0$,
				    we have \[ \lim_{n \to \infty} \mathbb{P}_{\boldsymbol{e}_{i}}\bigg(  \frac{1}{\langle \boldsymbol{1}, \boldsymbol{Z}(n)\rangle} \sum_{x \in X_{n}} \mathds{1}_{\{L^{x}(n) \in F\}}  > \varepsilon \ \bigg{\vert} \ \Omega_{\mathrm{surv}}(n) \bigg) =0.   \] 
				\end{remark}
				
				\begin{proof}[Proof of Theorem~\ref{theolargedevtimeaverage}]
					For fixed $\rho >0$, we consider the set $S= \{ \boldsymbol{\nu} \in  \mathcal{P}([d]) : d(\boldsymbol{\nu}, \boldsymbol{\alpha})\geq \rho  \}$, the complement of the open $\rho$-neighbourhood of $\boldsymbol{\alpha}$. It then suffices to show that \[\limsup_{n \to \infty} \frac{1}{n} \log \Big[\mathbb{P}_{\boldsymbol{e}_{i}}\Big( \sum_{x \in X_{n}} \mathds{1}_{\{L^{x}(n) \in S\}} \geq 1 \ \Big{\vert} \ \Omega_{\mathrm{surv}}(n) \Big)\Big] \leq  - \inf_{\boldsymbol{\nu} \in F}J_{\mathbf{P}}(\boldsymbol{\nu}). \]
					
					This will follow from Proposition~\ref{HOGtheo4.1} and the Donsker-Varadhan large deviation principle (see e.g. \cite[Thm.~4.1.43]{deuschel2001large} or \cite[Thm.~IV.7]{hollander2000large}). In order to apply Proposition~\ref{HOGtheo4.1}, we first use Markov's inequality to find 
					\begin{align*} \mathbb{P}_{\boldsymbol{e}_{i}}\Big( \sum_{x \in X_{n}} \mathds{1}_{\{L^{x}(n) \in S\}} \geq 1 \ \Big{\vert} \ \Omega_{\mathrm{surv}}(n) \Big) &\leq \frac{\mathbb{P}_{\boldsymbol{e}_{i}}\big(  \sum_{x \in X_{n}} \mathds{1}_{\{L^{x}(n) \in S\}} \geq 1\big)}{\mathbb{P}_{\boldsymbol{e}_{i}}(\Omega_{\mathrm{surv}}(n))} \\
						&\leq \frac{1}{\mathbb{P}_{\boldsymbol{e}_{i}}(\Omega_{\mathrm{surv}}(n))} \mathbb{E}_{\boldsymbol{e}_{i}}\Bigg[ \sum_{x \in X_{n}} \mathds{1}_{\{L^{x}(n) \in S\}} \Bigg]. 
					\end{align*} Now, for $F(X[0,n], x[0,n])= u_{i}u_{\sigma(x_{n})}^{-1} \mathds{1}_{\{L^{x_{n}}(n) \in S\}}$, Proposition~\ref{HOGtheo4.1} yields \[ \mathbb{E}_{\boldsymbol{e}_{i}}\Bigg[\sum_{x \in X_{n}} \mathds{1}_{\{L^{x}(n) \in S\}} \Bigg] = \widehat{\mathbb{E}}_{\boldsymbol{e}_{i}}^{*} \big[ u_{i}u_{\sigma(\xi_{n})}^{-1} \mathds{1}_{\{L^{\xi_{n}}(n) \in S\}}     \big]   \leq \max_{j} \frac{u_{i}}{u_{j}} \widehat{\mathbb{P}}_{\boldsymbol{e}_{i}}^{*} ( L^{\xi_{n}}(n) \in S ),   \] and so \[  \mathbb{P}_{\boldsymbol{e}_{i}}\Big(  \sum_{x \in X_{n}} \mathds{1}_{\{L^{x}(n) \in S\}} \geq 1 \ \Big{\vert} \ \Omega_{\mathrm{surv}}(n) \Big) \leq \frac{1}{\mathbb{P}_{\boldsymbol{e}_{i}}(\Omega_{\mathrm{surv}}(n))} \max_{j} \frac{u_{i}}{u_{j}} \widehat{\mathbb{P}}_{\boldsymbol{e}_{i}}^{*} ( L^{\xi_{n}}(n) \in S ).   \]  
					
					Taking logarithms on both sides of the inequality and dividing by $n$, we find \begin{align*}
						\frac{1}{n} \log  \Big[\mathbb{P}_{\boldsymbol{e}_{i}} & \Big(  \sum_{x \in X_{n}} \mathds{1}_{\{L^{x}(n) \in S\}}  \geq 1 \ \Big{\vert} \ \Omega_{\mathrm{surv}}(n) \Big)\Big] \\
						& \leq \frac{1}{n} \log \Big[ \frac{1}{\mathbb{P}_{\boldsymbol{e}_{i}}(\Omega_{\mathrm{surv}}(n))}  \Big] + \frac{1}{n} \log \Big[ \max_{j} \frac{u_{i}}{u_{j}} \Big]
						+ \frac{1}{n} \log \Big[ \widehat{\mathbb{P}}_{\boldsymbol{e}_{i}}^{*} ( L^{\xi_{n}}(n) \in S )  \Big].
					\end{align*} By \eqref{eqlimitOmegasurvn}, the first summand goes to $0$ for $n \to \infty$, as does the second. For the third, the Donsker--Varadhan large deviation principle yields \[ \limsup_{n \to \infty} \frac{1}{n} \log \big(\widehat{\mathbb{P}}_{\boldsymbol{e}_{i}}^{*}( L^{\xi_{n}}(n) \in S) \big) \leq - \inf_{\boldsymbol{\nu} \in S}J_{\mathbf{P}}(\boldsymbol{\nu}), \] and the statement follows. \qedhere
				\end{proof}
				

				\section{Typical ancestral type evolution}
				
				To set up our last result, we let $\Sigma = A(\mathbb{Z}, [d])$ denote the space of all doubly infinite sequences taking values in $[d]$. The time shift $\vartheta_{m}$ on $\Sigma$ is defined by \[ \vartheta_{m}s(n) = s(n+m), \quad n,m \in \mathbb{Z}, s \in \Sigma.  \] We write $\mathcal{P}_{\Theta}(\Sigma)$ for the set of all probability measures on $\Sigma$ that are invariant under the shift group $\Theta = (\vartheta_{m})_{m \in \mathbb{Z}}$ equipped with the weak topology.\\
				
				Next we introduce the time-averaged type-evolution process of an individual in the population tree. For $n \in \mathbb{N}$ and $x \in X_{n}$ we set \[ \sigma(x)_{n, \text{per}}(m) \coloneq \sigma(x(m_{n})), \quad m \in \mathbb{Z},  \] where $m_{n}$ is the unique number in $[n-1]_{0}$ with $m \equiv m_{n} \mod n$. That is, $\sigma(x)_{n, \text{per}}$ is the periodically continued type history of $x$ up to time $n$. The time-averaged type evolution of $x$ is then described by the \emph{empirical type evolution process} \begin{equation} R^{x}(n)= \frac{1}{n} \sum_{k=0}^{n-1} \delta_{\vartheta_{k}\sigma(x)_{n, \text{per}}} \in \mathcal{P}_{\Theta}(\Sigma). \label{eqdefR}
				\end{equation}
				
				We are interested in the typical behaviour of $R^{x}(n)$ when $x$ is picked at random from $X_{n}$, the population at time $n$.

				\begin{theorem}\label{theolargedevtypicalancestral}
					Let $\boldsymbol{\mu}  \in \mathcal{P}_{\Theta}(\Sigma)$ denote the distribution of the stationary (doubly infinite) retrospective mutation chain $(\sigma(n))_{n \in \mathbb{Z}}$ with transition matrix $\mathbf{P}$ as in Definition~\ref{defretrospectivemutationchain} and invariant distribution $\boldsymbol{\alpha}$. Then, for every closed $F \subset \mathcal{P}_{\Theta}(\Sigma)$ with $\boldsymbol{\mu} \notin F$,
					\[\limsup_{n \to \infty} \frac{1}{n} \log \Big[\mathbb{P}_{\boldsymbol{e}_{i}}\Big( \sum_{x \in X_{n}} \mathds{1}_{\{R^{x}(n) \in F\}} > 0 \ \Big{\vert} \ \Omega_{\mathrm{surv}}(n) \Big)\Big] \leq  - \inf_{\boldsymbol{\nu} \in F} H_{\mathbf{P}}(\boldsymbol{\nu})  \] with the rate function \[ H_{\mathbf{P}}(\boldsymbol{\nu}) = \sup_{n \in \mathbb{N}} \frac{H(\boldsymbol{\nu}|_{[n-1]_{0}}; \boldsymbol{\mu}|_{[n-1]_{0}})}{n},   \] where $\boldsymbol{\nu}|_{[n-1]_{0}}$ and $\boldsymbol{\mu}|_{[n-1]_{0}}$ denote the restrictions of $\boldsymbol{\nu}$ and $\boldsymbol{\mu}$ to times $[n-1]_{0}$ and $H(\boldsymbol{\nu}|_{[n-1]_{0}}; \boldsymbol{\mu}|_{[n-1]_{0}})$ is their relative entropy in the sense of \cite[Lem.~3.2.13]{deuschel2001large}.
				\end{theorem}
				
				\begin{proof}
					We apply the process level version of the Donsker-Varadhan large deviation principle (see e.g. \cite[Thm.~4.4.12]{deuschel2001large} or \cite[Thm.~5.3]{varadhan1988large}). The statement then follows from repeating the exact argument from the proof of Theorem~\ref{theolargedevtimeaverage}, replacing $L^{x}(n)$ and $\boldsymbol{\alpha}$ by $R^{x}(n)$ and $\boldsymbol{\mu}$.
				\end{proof}

				\begin{remark}\label{theotypicalancestraltypeevolution}
				In analogy with Remark \ref{theotimeaverageancestraltypes}, Theorem \ref{theolargedevtypicalancestral} implies a weaker analogue of \cite[Thm.~5.1]{georgii2003supercritical}, which says that, for all individuals $x \in X_{n}$ up to an asymptotically negligible fraction, the time-averaged ancestral type evolution process $R^{x}(n)$ is close to $\boldsymbol{\mu}$ in the weak topology: for each closed $F \subset \mathcal{P}_{\Theta}(\Sigma)$ with $\boldsymbol{\mu} \notin F$ and every $\varepsilon>0$, we have \[ \lim_{n \to \infty} \mathbb{P}_{\boldsymbol{e}_{i}}\bigg(  \frac{1}{\langle \boldsymbol{1}, \boldsymbol{Z}(n)\rangle} \sum_{x \in X_{n}} \mathds{1}_{\{R^{x}(n) \in F\}}  > \varepsilon \ \bigg{\vert} \ \Omega_{\mathrm{surv}}(n) \bigg) =0.   \] 
				\end{remark}
			\section{Acknowledgements}	
			Funded by the Deutsche Forschungsgemeinschaft (DFG, German Research Foundation) - Project-ID 317210226 - SFB 1283.
				
				\bibliographystyle{abbrvnat}

\begin{thebibliography}{26}
\providecommand{\natexlab}[1]{#1}
\providecommand{\url}[1]{\texttt{#1}}
\expandafter\ifx\csname urlstyle\endcsname\relax
  \providecommand{\doi}[1]{doi: #1}\else
  \providecommand{\doi}{doi: \begingroup \urlstyle{rm}\Url}\fi

\bibitem[Arratia et~al.(2019)Arratia, Goldstein, and Kochman]{arratia2019size}
R.~Arratia, L.~Goldstein, and F.~Kochman.
\newblock Size bias for one and all.
\newblock \emph{Probab. Surv.}, 16:\penalty0 1--61, 2019.
\newblock \doi{10.1214/13-ps221}.

\bibitem[Athreya and Ney(1974)]{athreya1974functionals}
K.~Athreya and P.~Ney.
\newblock Functionals of critical multitype branching processes.
\newblock \emph{Ann. Probab.}, 2:\penalty0 339--343, 1974.
\newblock \doi{10.1214/aop/1176996716}.

\bibitem[Athreya and Ney(1972)]{athreya2004branching}
K.~B. Athreya and P.~E. Ney.
\newblock \emph{Branching Processes}.
\newblock Springer, New York, 1972.
\newblock Reprint Dover, New York, 2004.

\bibitem[Billingsley(1999)]{billingsley2013convergence}
P.~Billingsley.
\newblock \emph{Convergence of Probability Measures}.
\newblock Wiley, New York, 2nd edition, 1999.
\newblock \doi{10.1002/9780470316962}.

\bibitem[Cheek and Johnston(2023)]{CheekJohnston2023}
D.~Cheek and S.~Johnston.
\newblock Ancestral reproductive bias in branching processes.
\newblock \emph{J. Math. Biol.}, 86:\penalty0 paper no. 70, 2023.
\newblock \doi{10.1007/s00285-023-01907-7}.

\bibitem[Danka and Pap(2016)]{Danka_Pap_16}
P.~Danka and G.~Pap.
\newblock Asymptotic behavior of critical indecomposable multi-type branching
  processes with immigration.
\newblock \emph{ESAIM Prob. Stat.}, 20:\penalty0 238--260, 2016.
\newblock \doi{10.1051/ps/2016010}.

\bibitem[den Hollander(2000)]{hollander2000large}
F.~den Hollander.
\newblock \emph{Large Deviations}.
\newblock American Mathematical Society, Providence, RI, 2000.
\newblock \doi{10.1007/s00440-009-0235-5}.

\bibitem[Deuschel and Stroock(1989)]{deuschel2001large}
J.-D. Deuschel and D.~W. Stroock.
\newblock \emph{Large Deviations}.
\newblock Academic Press, Boston, MA, 1989.

\bibitem[Georgii and Baake(2003)]{georgii2003supercritical}
H.-O. Georgii and E.~Baake.
\newblock Supercritical multitype branching processes: the ancestral types of
  typical individuals.
\newblock \emph{Adv. Appl. Probab.}, 35:\penalty0 1090--1110, 2003.
\newblock \doi{10.1239/aap/1067436336}.

\bibitem[Harris(1963)]{harris1963theory}
T.~E. Harris.
\newblock \emph{The Theory of Branching Processes}.
\newblock Springer, Berlin, 1963.
\newblock Corrected reprint Dover, New York, 2002.

\bibitem[Jagers(1989)]{jagers1989}
P.~Jagers.
\newblock General branching processes as {M}arkov fields.
\newblock \emph{Stoch. Processes Appl.}, 32:\penalty0 183--242, 1989.

\bibitem[Jagers(1993)]{jagers1992stabilities}
P.~Jagers.
\newblock Stabilities and instabilities in population dynamics.
\newblock In V.~Kalashniko and V.~M. Zolatarev, editors, \emph{Stability
  Problems for Stochastic Models ({S}uzdal, 1991)}, pages 58-- 69. Springer,
  Berlin, 1993.
\newblock \doi{10.1007/BFb0084482}.

\bibitem[Jagers and Nerman(1996)]{jagers1996asymptotic}
P.~Jagers and O.~Nerman.
\newblock The asymptotic composition of supercritical multi-type branching
  populations.
\newblock In J.~Az{\'e}ma and M.~Emery, editors, \emph{S\'eminaire de
  Probabilit\'es XXX}, pages 40--54. Springer, Berlin, 1996.
\newblock \doi{10.1007/BFb0094640}.

\bibitem[Joffe and Spitzer(1967)]{joffe1967multitype}
A.~Joffe and F.~Spitzer.
\newblock On multitype branching processes with {$\rho \leq 1$}.
\newblock \emph{J. Math. Anal. Appl.}, 19:\penalty0 409--430, 1967.
\newblock \doi{10.1016/0022-247X(67)90001-7}.

\bibitem[Karlin and Taylor(1975)]{karlintaylor1975}
S.~Karlin and H.~Taylor.
\newblock \emph{A First Course in Stochastic Processes}.
\newblock Academic Press, Boston, 2nd edition, 1975.
\newblock \doi{10.1016/C2009-1-28569-8}.

\bibitem[Keating et~al.(2022)Keating, Gleeson, and O'Sullivan]{Keating2022}
L.~A. Keating, J.~P. Gleeson, and D.~J.~P. O'Sullivan.
\newblock Multitype branching process method for modeling complex contagion on
  clustered networks.
\newblock \emph{Phys. Rev. E}, 105:\penalty0 034306, 12 pages, 2022.
\newblock \doi{10.1103/physreve.105.034306}.

\bibitem[Kesten and Stigum(1966)]{KestenStigum1966}
H.~Kesten and B.~Stigum.
\newblock A limit theorem for multidimensional galten--watson processes.
\newblock \emph{Ann. Math. Statist.}, 37:\penalty0 1211--1233, 1966.
\newblock \doi{10.1214/aoms/1177699266}.

\bibitem[Kurtz et~al.(1997)Kurtz, Lyons, Pemantle, and
  Peres]{kurtz1997conceptual}
T.~Kurtz, R.~Lyons, R.~Pemantle, and Y.~Peres.
\newblock A conceptual proof of the {K}esten--{S}tigum theorem for multi-type
  branching processes.
\newblock In K.~B. Athreya and P.~Jagers, editors, \emph{Classical and Modern
  Branching Processes}, pages 181--185. Springer, Minneapolis, MN, 1997.
\newblock \doi{10.1007/978-1-4612-1862-3_14}.

\bibitem[Mullikin(1963)]{mullikin1963limiting}
T.~W. Mullikin.
\newblock Limiting distributions for critical multitype branching processes
  with discrete time.
\newblock \emph{Trans. Amer. Math. Soc.}, 106:\penalty0 469--494, 1963.
\newblock \doi{10.2307/1993755}.

\bibitem[P\'{e}nisson(2014)]{Penisson_2014}
S.~P\'{e}nisson.
\newblock Estimation of the infection parameter of an epidemic modeled by a
  branching process.
\newblock \emph{Electron. J. Stat.}, 8:\penalty0 2158--2187, 2014.
\newblock \doi{10.1214/14-EJS948}.

\bibitem[Singh et~al.(2014)Singh, Schneider, and Myers]{Sarabjeet2014}
S.~Singh, D.~J. Schneider, and C.~R. Myers.
\newblock Using multitype branching processes to quantify statistics of disease
  outbreaks in zoonotic epidemics.
\newblock \emph{Phys. Rev. E}, 89:\penalty0 032702, 2014.
\newblock \doi{10.1103/PhysRevE.89.032702}.

\bibitem[Smerlak(2021)]{Smerlak2021}
M.~Smerlak.
\newblock Quasi-species evolution maximizes genotypic reproductive value (not
  fitness or flatness).
\newblock \emph{J. Theor. Biol.}, 522:\penalty0 110699, 4 pages, 2021.
\newblock \doi{10.1016/j.jtbi.2021.110699}.

\bibitem[Spencer and O'Neill(2011)]{Spencer_ONeill_2011}
S.~E.~F. Spencer and P.~D. O'Neill.
\newblock The probability of containment for multitype branching process models
  for emerging epidemics.
\newblock \emph{J. Appl. Probab.}, 48:\penalty0 173--188, 2011.
\newblock \doi{10.1239/jap/1300198143}.

\bibitem[Sughiyama and Kobayashi(2017)]{SughiyamaKobayashi2017}
Y.~Sughiyama and T.~Kobayashi.
\newblock Steady-state thermodynamics for population growth in fluctuating
  environments.
\newblock \emph{Phys. Rev. E}, 95:\penalty0 012131, 9 pages, 2017.
\newblock \doi{10.1103/PhysRevE.95.012131}.

\bibitem[Varadhan(1988)]{varadhan1988large}
S.~R.~S. Varadhan.
\newblock Large deviations and applications.
\newblock In P.-L. Hennequin, editor, \emph{\'Ecole d'\'Et\'e{} de
  Probabilit\'es de Saint-Flour XV--XVII, 1985--87}, pages 1--49. Springer,
  Berlin, 1988.
\newblock \doi{10.1007/BFb0086178}.

\bibitem[Yakovlev and Yanev(2009)]{Yakovlev2009}
A.~Y. Yakovlev and N.~M. Yanev.
\newblock {Relative frequencies in multitype branching processes}.
\newblock \emph{Ann. Appl. Probab.}, 19:\penalty0 1--14, 2009.
\newblock \doi{10.1214/08-AAP539}.

\end{thebibliography}

			\end{document}